\documentclass[12pt]{amsart}
\usepackage{amssymb,latexsym,eufrak,amsmath,amscd, graphicx}
\usepackage[colorlinks=true,pagebackref,hyperindex]{hyperref}

\usepackage{amsfonts}
\usepackage{amscd}

\usepackage{xypic}

\renewcommand{\phi}{\varphi}
\renewcommand{\epsilon}{\varepsilon}
\renewcommand{\theta}{\vartheta}

\def\ZZ{{\mathbf Z}}

\def\FF{{\mathbf F}}

\def\AAA{{\mathbf A}}
\def\RR{{\mathbf R}}
\def\QQ{{\mathbf Q}}

\def\FF{{\mathbf F}}

\def\cJ{\mathcal{J}}

\def\cF{\mathcal{F}}

\def\cG{\mathcal{G}}
\def\cL{\mathcal{L}}
\def\cO{\mathcal{O}}

\def\fra{\mathfrak{a}}
\def\frb{\mathfrak{b}}

\def\frm{\mathfrak{m}}

\DeclareMathOperator{\codim}{codim}

 \DeclareMathOperator{\Spec}{Spec}

\newtheorem{lemma}{Lemma}[section]
\newtheorem{theorem}[lemma]{Theorem}
\newtheorem{corollary}[lemma]{Corollary}
\newtheorem{proposition}[lemma]{Proposition}

\newtheorem{conjecture}[lemma]{Conjecture}

\theoremstyle{definition}

\newtheorem{remark}[lemma]{Remark}

\newtheorem{example}[lemma]{Example}

\theoremstyle{remark}
\newtheorem*{remark*}{Remark}
\newtheorem*{note*}{Note}

\begin{document}

\title{Ordinary varieties and the comparison between multiplier ideals
and test ideals}

\thanks{2000\,\emph{Mathematics Subject Classification}.
 Primary 13A35; Secondary 14F18, 14F30.
\newline Musta\c{t}\u{a} was partially supported by
 NSF grant DMS-0758454 and
  a Packard Fellowship. Srinivas was partially supported by a J.C.
Bose Fellowship of the Department of Science and Technology, India.}

\author[M.~Musta\c{t}\u{a}]{Mircea~Musta\c{t}\u{a}}
\address{Department of Mathematics, University of Michigan,
Ann Arbor, MI 48109, USA}
\email{{mmustata@umich.edu}}

\author[V.~Srinivas]{Vasudevan~Srinivas}
\address{School of Mathematics, Tata Institute of Fundamental Research,
Homi Bhabha Road, Colaba, Mumbai 400005, India}
\email{{srinivas@math.tifr.res.in}}

\begin{abstract}
We consider the following conjecture: if $X$ is a smooth and irreducible  $n$-dimensional projective variety over a field $k$ of characteristic zero, then there is a dense set  of reductions $X_s$ to positive characteristic such that the action of the Frobenius morphism on $H^n(X_s,\cO_{X_s})$
is bijective. There is another conjecture relating certain invariants of singularities in characteristic zero (the multiplier ideals) with invariants in positive characteristic (the test ideals). We prove that the former conjecture implies the latter one in the case of ambient nonsingular varieties. 
\end{abstract}

\maketitle

\markboth{M.~MUSTA\c{T}\u{A} AND V.~SRINIVAS}{ORDINARY VARIETIES, MULTIPLIER IDEALS,
AND TEST IDEALS}

\section{Introduction}
It has been known for about thirty years that there are close connections between classes of singularities that appear in birational geometry, and such classes that appear in commutative algebra, and more precisely, in tight closure theory. Recall that in birational geometry, singularities are typically described in terms of a suitable resolution of singularities. On the other hand, tight closure theory describes the singularities in positive characteristic 
in terms of the action of the Frobenius morphism. The connection between the two points of view
is very rich, but still remains somewhat mysterious.

The best known example of such a connection concerns rational singularities: it says that a variety has rational singularities if and only if it has $F$-rational type
($F$-rationality is a notion defined in positive characteristic via the tight closure of parameter ideals).
More precisely, suppose that  $X$ is defined over a field $k$ 
of characteristic zero, and consider a model of $X$ defined over an algebra $A$ of finite type 
over $\ZZ$. For every closed point $s\in\Spec A$ consider the corresponding reduction $X_s$
to positive characteristic. Then $X$ has rational singularities if and only if there is an open subset
$U$ of $\Spec A$ such that $X_s$ has $F$-rational singularities for every closed point $s\in U$ (the ``if" part was proved in \cite{Smith},
while the ``only if" part was proved independently in \cite{Hara} and \cite{MS}).

Other classes of singularities behave in the same fashion: see \cite{HW}
for the comparison between Kawamata log terminal and strongly  $F$-regular singularities. On the other hand,
a more subtle phenomenon relates, for example, log canonical and $F$-pure singularities.
It is known that if there is a (Zariski) dense set of closed points $S\subset\Spec A$ such that $X_s$ has $F$-pure singularities for all $s\in S$,
then $X$ has log canonical singularities (see \cite{HW}). The converse, however, is widely open,
and in general the set of closed points $s\in\Spec A$ for which $X_s$ has $F$-pure singularities does not contain an open subset, even when it is dense. Furthermore, examples have made it clear that there are some subtle arithmetic phenomena involved.

The main goal of our paper is to consider an arithmetic-geometric conjecture, and show that it implies 
a similar such connection, between
multiplier ideals (invariants in characteristic zero) and test ideals (invariants in characteristic $p$). 
We believe that this puts in a new perspective the correspondence between the two sets of invariants, and hopefully points to a possible way of proving this correspondence.

\begin{conjecture}\label{conj_introd}
Let $X$ be a smooth, connected $n$-dimensional projective variety over an algebraically closed field
$k$ of characteristic zero. Given a model of $X$ over a $\ZZ$-algebra of finite type $A$, 
contained in $k$, there is a dense set of closed points $S\subseteq \Spec A$ such that the action induced by Frobenius on
$H^n(X_s,\cO_{X_s})$ is bijective for every $s\in S$.
\end{conjecture}

As we show, in the above conjecture it is enough to consider the case $k=\overline{\QQ}$
(see Proposition~\ref{reduction_to_number_field}).
We mention that it is expected that under the assumptions in the conjecture, there is a dense set of closed points $S\subseteq \Spec A$, such that for every $s\in S$, the smooth projective
variety $X_s$ over $k(s)$ is ordinary in the sense of \cite{BK}. One can show that this condition 
implies that the action induced by Frobenius on each cohomology group
$H^i(X_s,\cO_{X_s})$ is bijective. On the other hand, we hope that the property in 
Conjecture~\ref{conj_introd} would be easier to prove than the stronger property of being ordinary.

Before stating the consequence of Conjecture~\ref{conj_introd} to the relation between multiplier ideals and test ideals, let us
recall the definitions of these ideals. Since our main result only deals with nonsingular ambient
varieties, we review these concepts in this special case.
Let $Y$ be a nonsingular, connected variety defined over an algebraically closed field of characteristic zero,
and suppose that $\fra$ is a nonzero ideal on 
$Y$. Recall that a log resolution of $(Y,\fra)$ is a projective birational morphism 
$\pi\colon X\to Y$, with $X$ nonsingular and $\fra\cdot\cO_X=\cO_X(-G)$, with $G$ a divisor, such that there is a simple normal crossings divisor $E$ on $X$, with
both $G$ and $K_{X/Y}$ supported on $E$. Here $K_{X/Y}$ is the relative canonical divisor. 
Such resolutions exist by Hironaka's theorem, since $Y$ lives in characteristic zero. 
The multiplier ideal of $\fra$ of exponent $\lambda\geq 0$ is the ideal
$$\cJ(Y,\fra^{\lambda}):=\pi_*\cO_X(K_{X/Y}-\lfloor\lambda G\rfloor),$$
where for any $\RR$-divisor $E$, we denote by $\lfloor E\rfloor$ its round-down. 
It is a general fact that the definition is independent of the given resolution. 
These ideals have recently found many striking applications in birational geometry, mostly
due to their connection with vanishing theorems, see \cite{positivity}.

In positive characteristic, Hara and Yoshida \cite{HY} introduced the notion of
(generalized) test ideal, relying on a generalization of the theory of tight closure.  
In this paper we use an equivalent definition due to Schwede \cite{Schwede}.
This definition is particularly transparent in the case of an ambient nonsingular variety, when it is an immediate 
consequence of the description in \cite{BMS}.

Suppose that $Y$ is a nonsingular, connected variety over a perfect field $L$ of characteristic $p>0$, 
and $\fra$ is an ideal on $Y$. The Cartier isomorphism induces a surjective 
$\cO_X$-linear map
$t_Y\colon F_*\omega_Y\to\omega_Y$, where $F$ is the absolute Frobenius morphism.
Iterating this $e$ times gives $t_Y^e\colon F^e_*\omega_Y\to\omega_Y$.
For any ideal $\frb$ on $Y$, and for every
$e\geq 1$, the ideal $\frb^{[1/p^e]}$ is defined by 
$t_Y^e(F^e_*(\frb\cdot\omega_Y))=\frb^{[1/p^e]}\cdot\omega_Y$.
Given any $\lambda\geq 0$, it is easy to see that the sequence of ideals
$\left((\fra^{\lceil \lambda p^e\rceil})^{[1/p^e]}\right)_{e\geq 1}$ is nondecreasing, and therefore
it stabilizes by the Noetherian property. The limit is the test ideal $\tau(Y,\fra^{\lambda})$.
For a discussion of various analogies between test ideals and multiplier ideals we refer to
\cite{HY}. The following is the main conjecture relating multiplier ideals and test ideals.

\begin{conjecture}\label{conj_introd2}
Let $Y$ be a nonsingular, connected variety over an algebraically closed
 field $k$ of characteristic zero, and $\fra$
a nonzero ideal on $Y$. Given a model for $Y$ and $\fra$ defined over a $\ZZ$-algebra of finite type $A$, contained in $k$, there is a dense set of closed points $S\subset \Spec A$, such that 
\begin{equation}\label{eq_conj_introd2}
\tau(Y_s,\fra_s^{\lambda})=\cJ(Y,\fra^{\lambda})_s
\end{equation}
for all $s\in S$ and all $\lambda\geq 0$.
Furthermore, if we have finitely many pairs as above
$(Y^{(i)}, \fra^{(i)})$, and corresponding models over $\Spec A$, then there is
a dense open subset of closed points in $\Spec A$ such that (\ref{eq_conj_introd2})
holds for each of these pairs.
\end{conjecture}

Two things are known: first, under the assumptions in the conjecture, there is an open
subset of closed points in $\Spec A$ for which the inclusion ``$\subseteq$" in (\ref{eq_conj_introd2}) holds for all $\lambda$. This was proved in
\cite{HY}, and is quite elementary (we give a variant of the argument in \S 3, using the equivalent definition
in \cite{Schwede}). 
A deeper result, also proved in \cite{HY}, says that for a \emph{fixed} $\lambda$,
there is an open subset of closed points $s\in\Spec A$ such that equality holds
in (\ref{eq_conj_introd2}) for this $\lambda$. This relies on the same kind of arguments as in
\cite{Hara} and \cite{MS}, using the action of Frobenius on the de Rham complex, following
\cite{DI}. The key fact in the above conjecture is that we require the equality to hold
for all $\lambda$ at the same time. We mention that these two known results generalize the fact that
$(Y,\fra^{\lambda})$ is Kawamata log terminal if and only if for an open (or just dense) set of closed points
$S\subset \Spec A$ the pair $(Y_s,\fra_s^{\lambda})$ is strongly $F$-regular for all $s\in S$. The following is our
main result.

\begin{theorem}\label{thm_introd}
If Conjecture~\ref{conj_introd} holds, then Conjecture~\ref{conj_introd2} holds as well.
\end{theorem}

It is easy to reduce the assertion in Conjecture~\ref{conj_introd2} to the case when $Y$ is affine
and $\fra$ is a principal ideal $(f)$.
The usual approach for comparing the multiplier ideals of $\fra$ with the test ideals of a reduction mod $p$
of $\fra$ is to start
with a log resolution of $\fra$. Our key point is to start instead by doing semistable reduction.
This allows us to reduce at the end of the day to understanding a certain reduced divisor with simple normal crossings on a nonsingular variety. 

One can formulate Conjecture~\ref{conj_introd2} in a more general setting. For example, one can only assume that $Y$ is normal and $\QQ$-Gorenstein, or even more generally, work with a pair $(Y,D)$ such that
$K_Y+D$ is $\QQ$-Cartier. Furthermore, one can start with several ideals $\fra_1,\ldots,\fra_r$,
and consider mixed multiplier ideals and test ideals. However, our method based on semistable 
reduction does not allow us to handle at present these more general versions of Conjecture~\ref{conj_introd2}.

The paper is organized as follows. In the next section we recall some general facts about $p$-linear maps of vector spaces over perfect fields, and review the general setting for reducing from characteristic zero to positive characteristic. In \S\ref{test_and_multiplier} we recall the definition and some useful properties of multiplier ideals  and test ideals. 
While in our main result we consider a nonsingular ambient variety, at an intermediate step 
we also need to work on a singular variety. Therefore our treatment of multiplier ideals and test ideals in \S\ref{test_and_multiplier} is done in this general setting. 
In Section~\ref{conjectural_connection} we 
state and discuss a more general version of Conjecture~\ref{conj_introd2}.
Section~\ref{conjecture_Frobenius_action} is devoted to a discussion of Conjecture~\ref{conj_introd},
and to several consequences that would be needed later. In the last Section~\ref{connection_conjectures} we prove our main result, showing that Conjecture~\ref{conj_introd} implies Conjecture~\ref{conj_introd2}.

\subsection*{Acknowledgment} 
We are indebted to Bhargav Bhatt, H\'{e}l\`{e}ne Esnault, and
Johannes Nicaise 
for several inspiring discussions. We would also like to thank Karl Schwede and the 
anonymous referee, whose thoughtful comments
helped improve the paper. Part of this work was done during the second author's visit
to Ann Arbor. We are grateful to University of Michigan and to the David and Lucile Packard foundation for making this visit possible.

\section{A review of basic facts}\label{review}

In this section we recall some well-known facts that will frequently come up during the rest of the paper. In particular, we discuss the general setting, and set the notation for reduction mod $p$.

\subsection{$p$-linear maps on vector spaces}\label{p_linear}
Let $k$ be a perfect field of characteristic $p>0$, and let $V$ be a finite-dimensional
vector space over $k$. Let $\phi\colon V\to V$ be a $p$-linear map, that is, a morphism of abelian groups such that $\phi(au)=a^p\phi(u)$ for all $a\in k$ and $u\in V$. The following properties of such a map are well-known; for a proof see for example \cite[Lemma~3.3]{CL}. 

The vector space $V$ can be uniquely decomposed as a direct sum of subspaces preserved by 
$\phi$, 
$V=V_{\rm ss}\oplus V_{\rm nil}$, where 
\begin{enumerate}
\item[1)] $\phi$ is nilpotent on $V_{\rm nil}$, that is, $\phi^N=0$ for some $N$.
\item[2)] $\phi$ is bijective on $V_{\rm ss}$.
\end{enumerate}
One says that $\phi$ is \emph{semisimple} if $V=V_{\rm ss}$. This is equivalent with $\phi$
being injective, or equivalently, surjective.

\begin{example}\label{ex1_1}
If $k$ is a finite field with $p^e$ elements, then $\phi^e$ is a $k$-linear map. In this case $\phi$
is semisimple if and only if $\phi^e$ is an isomorphism.
\end{example}

If $\phi$ is as above, and
$k'$ is a perfect  field extension of $k$, then we get an induced
$p$-linear map $\phi'\colon V'\to V'$, where
$V'=V\otimes_kk'$. This is given by
$\phi'(v\otimes\lambda)=\phi(v)\otimes\lambda^p$. We have
$V'_{\rm ss}=V_{\rm ss}\otimes_kk'$ and $V'_{\rm nil}
=V_{\rm nil}\otimes_kk'$. In particular, $\phi'$ is semisimple if and only if
$\phi$ is semisimple.

These considerations apply, in particular, if we take $k'=\overline{k}$, an
algebraic closure of $k$. If $\overline{\phi}\colon
\overline{V}\to \overline{V}$ is the induced $p$-linear map over $\overline{k}$, then
$\overline{V}^{\overline{\phi}=1}:=\{u\in \overline{V}\mid \overline{\phi}(u)=u\}$
is an $\FF_p$-vector subspace of $\overline{V}$ such that 
\begin{equation}\label{eq1_1}
\overline{V}_{\rm ss}=\overline{V}^{\overline{\phi}=1}\otimes_{\FF_p}\overline{k}.
\end{equation}
In particular, we have $\dim_{\FF_p}(\overline{V}^{\overline{\phi}=1})\leq\dim_k(V)$,
with equality if and only if $\phi$ is semisimple. 

Note that the morphism of abelian groups $1-\overline{\phi}$ is surjective on 
$\overline{V}_{\rm ss}$ by (\ref{eq1_1}), and it is clearly bijective on $\overline{V}_{\rm nil}$.
In particular, $1-\overline{\phi}$ is surjective, and its kernel is $\overline{V}^{\overline{\phi}=1}$.

\begin{example}\label{ex1_2}
Let $X$ be a complete scheme of finite type over $k$. The absolute Frobenius morphism $F\colon X\to X$
is the identity on the underlying topological space, and the corresponding  morphism of sheaves of rings $\cO_X\to\cO_X$ is given by $u\to u^p$.
Since $k$ is perfect,
$F$ is a finite morphism. It induces a $p$-linear map $F\colon H^i(X,\cO_X)\to H^i(X,\cO_X)$
for every $i\geq 0$. After extending the scalars to an algebraic closure $\overline{k}$, we obtain the corresponding $p$-linear map $F\colon H^i(X_{\overline{k}},\cO_{X_{\overline{k}}})\to 
H^i(X_{\overline{k}},\cO_{X_{\overline{k}}})$, where $X_{\overline{k}}=
X\times_{\Spec k}\Spec \overline{k}$ 
(note that in this case we still write $F$ instead of $\overline{F}$). 

On the other hand, we have the Artin-Schreyer sequence in the \'{e}tale topology
$$0\to \FF_p\to \cO_{X_{\overline{k}}} \overset{1-F}\to \cO_{X_{\overline{k}}} \to 0.$$
This induces exact sequences
$$0\to H^i_{\rm \acute{e}t}(X_{\overline{k}},\FF_p)\to H^i(X_{\overline{k}},\cO_{X_{\overline{k}}})
\overset{1-F}\to 
H^i(X_{\overline{k}},\cO_{X_{\overline{k}}})\to 0$$
for every $i\geq 0$. In particular, $F$ is semisimple on $H^i(X,\cO_X)$ if and only if 
$\dim_{\FF_p}H^i_{\rm \acute{e}t}(X_{\overline{k}},\FF_p)=h^i(X,\cO_X)$. 
\end{example}

\begin{remark}\label{rem_Frob_product}
Let $\phi\colon V\to V$ and $\psi\colon W\to W$ be $p$-linear maps as above.
Note that we have induced $p$-linear maps on $V\oplus W$ and $V\otimes W$, and
$(V\oplus W)_{\rm ss}=V_{\rm ss}\oplus W_{\rm ss}$ and $(V\otimes W)_{\rm ss}
=V_{\rm ss}\otimes W_{\rm ss}$.
\end{remark}

\begin{lemma}\label{rem_Frob_exact}
Let $\phi\colon V\to V$ be a $p$-linear map as above.
\begin{enumerate}
\item[i)] If $\phi$ is semisimple, and if $W$ is a linear subspace of $V$ such that
$\phi(W)\subseteq W$, then the induced $p$-linear maps on $W$ and $V/W$ are semisimple.
\item[ii)] If we have an exact sequence $V'\to V\to V''$, and $p$-linear maps
$\phi'\colon V'\to V'$ and $\phi''\colon V''\to V''$ that are compatible with $\phi$ in the obvious sense,
and if $\phi'$ and $\phi''$ are semisimple, then so is $\phi$. 
\end{enumerate}
\end{lemma}

\begin{proof}
If $\phi$ is bijective, then clearly the induced map on $W$ is injective, and the induced map
on $V/W$ is surjective. This implies the assertion in i). In order to prove ii), we use i) to reduce to the case when we have a short exact sequence 
$$0\to V'\to V\to V''\to 0.$$
In this case $\phi'$ and $\phi''$ being bijective implies $\phi$ is bijective by the 5-Lemma.
\end{proof}

\subsection{Reduction mod $p$}\label{reduction_mod_p}
We review the formalism for passing from characteristic zero to positive characteristic.
Let $k$ be a fixed field of characteristic zero. Given a scheme $X$ of finite type over $k$,
there is a subring $A\subset k$ of finite type over $\ZZ$, a scheme $X_A$ of finite type over $A$, and an
isomorphism $X\simeq X_A\times_{\Spec A}\Spec k$. 
Note that we may always replace $A$ by $A_a$,
for some nonzero $a\in A$, and $X_A$ by the corresponding open subscheme. 
It follows from Generic Flatness (see \cite[Theorem~14.4]{Eis}) 
that we may (and will) assume that $X_A$
is flat over $A$. We will refer to $X_A$ as a \emph{model} of $X$ over $A$.
 If $A$ and $B$ are two such rings, and if $X_A$ and $X_B$ 
are models of $X$ over $A$ and $B$, respectively, then there is a subring $C$ of $k$ containing both $A$ and $B$, finitely generated over $\ZZ$, and an isomorphism $X_A\times_{\Spec A}\Spec C\simeq X_B\times_{\Spec B}
\Spec C$ compatible after base-change to $\Spec k$ with the defining isomorphisms for
$X_A$ and $X_B$. 
Given a model $X_A$ for $X$ as above, and a point $s\in \Spec A$, we denote 
by $X_s$ the fiber of $X_A$ over $s$. 
This is a scheme of finite type over the residue field $k(s)$ of $s$.
Note that if $s$ is a closed point, then
$k(s)$ is a finite field. 

We will consider properties ${\mathcal P}$ of schemes of finite type over finite fields,
such that given a scheme $W$ of finite type over the finite field $k$, and a finite field extension
$k'$ of $k$,  ${\mathcal P}(W)$ holds if and only if ${\mathcal P}(W\times_{\Spec k}\Spec k')$
holds. 
With $X_A$ as above, we say that 
${\mathcal P}(X_s)$ holds for general closed points $s\in \Spec A$
if there is an open subset $U$ of $\Spec A$ such that ${\mathcal P}(X_s)$ holds for all
closed points  $s\in U$.
In this case, after replacing $A$ by a suitable localization $A_a$, we may assume that 
${\mathcal P}(X_s)$ holds for all closed points $s$. We will 
often be interested in properties that are expected to only hold for a dense set of closed points
$s\in\Spec A$.

\begin{remark}\label{enough_nonempty}
With ${\mathcal P}$ as above, note that both conditions 
\begin{enumerate}
\item[i)] ${\mathcal P}(X_s)$ holds for general closed points $s\in {\rm Spec} A$
\item[ii)] ${\mathcal P}(X_s)$ holds for a dense set of closed points $s\in {\rm Spec} A$
\end{enumerate}
are independent of the choice of a model. Indeed, if $\alpha\colon\Spec C\to\Spec A$ is induced by the inclusion
$A\subset C$ of finitely generated $\ZZ$-algebras, then $\alpha$ takes closed points to closed points, and the image of $\alpha$ contains a (dense) open subset. Furthermore, 
the image or inverse image of a dense subset has the same property. 

On the other hand, in order to show that ii) above holds, it is enough to show that for every model $X_A$, there is at least one closed point $s\in\Spec A$ such that ${\mathcal P}(X_s)$ holds.
\end{remark}

If $X_A$ is a model for $X$ as above, and if $\cF$ is a coherent sheaf on $X$, then
after possibly replacing $A$ by a larger ring we may assume that there is a coherent sheaf
$\cF_A$ on $X_A$ whose pull-back to $X$ is isomorphic to $\cF$. 
It follows from Generic Flatness that after replacing
$A$ by some localization $A_a$, we may (and will) assume that $\cF_A$
is flat over $A$. For a point $s\in\Spec A$, we denote by $\cF_s$ the restriction of
$\cF_A$ to the fiber over $s$. 

If
$\phi\colon \cF\to \cG$ is a morphism of coherent sheaves, after possibly enlarging $A$
we may assume that $f$ is induced by a morphism of sheaves $\phi_A\colon\cF_A\to\cG_A$.
In particular, for every point $s\in \Spec A$, we get an induced morphism
$\phi_s\colon\cF_s\to\cG_s$. 
Since we may assume that ${\rm Coker}(\phi_A)$ and ${\rm Im}(\phi_A)$ are flat over $A$, it follows
that we may assume that ${\rm Coker}(\phi_s)={\rm Coker}(\phi)_s$, ${\rm Im}(\phi_s)
={\rm Im}(\phi)_s$, and ${\rm Ker}(\phi_s)={\rm Ker}(\phi)_s$ for every point 
$s\in \Spec\,A$. In particular, if $\phi$ is injective
or surjective, then so are all $\phi_s$. It follows easily from this that if $\cF$ is an ideal, or if it is locally free, then so are all $\cF_s$ (as well as $\cF_A$). 

Given a morphism $f\colon X\to Y$ of schemes of finite type over $k$,
and models $X_A$ and $Y_B$ of $X$ and $Y$, respectively, after possibly enlarging
both $A$ and $B$ we may assume that $A=B$ and that 
$f$ is induced by a morphism 
$f_A\colon X_A\to Y_A$ of schemes over $A$.  If $s\in\Spec A$ is a point, then we get
a corresponding morphism $f_s\colon X_s\to Y_s$ of schemes over $k(s)$. 
If $f$ is either of the following: a closed (open) immersion, finite or projective, then we may assume that the same holds for $f_A$. In particular, the same will hold for all $f_s$ . 

Suppose now that $f\colon X\to Y$ is a proper morphism, and $\cF$ is a coherent sheaf on
$X$. If $f_A\colon X_A\to Y_A$ and $\cF_A$ are as above, arguing as in 
\cite[Section III. 12]{Hartshorne} one can show that $\cF_A$ satisfies \emph{generic base-change}.
In other words, after replacing $A$ by $A_a$ for some nonzero $a\in A$, we may assume that for all 
$s\in\Spec A$, the canonical morphism
$$\left(R^i(f_A)_*(\cF_A)\right)_s\to R^i(f_s)_*(\cF_s)$$
is an isomorphism. 

Given a model $X_A$ of $X$, it is easy to deduce from Noether's Normalization Theorem that
all fibers of $X_A\to\Spec A$ have dimension $\leq\dim(X)$. It follows from the Jacobian Criterion
for smoothness that if $X$ is an irreducible regular scheme, then we may assume that $X_A$ is smooth
over $\Spec A$ of relative dimension equal to $\dim(X)$. In particular, $X_s$ is smooth over 
$k(s)$ for every point $s\in\Spec A$. In general, $X_s$ might not be connected; however,
if we assume that $k$ is algebraically closed, then $X_s$ will be connected, since the
generic fiber of $X_A$ over $\Spec A$ is geometrically connected.

For simplicity, from now on we assume that $k$ is algebraically closed.
Suppose that $Y$ is an arbitrary reduced scheme over $k$, and let us consider a resolution of singularities of $Y$, that is, a projective birational morphism $f\colon X\to Y$, with $X$ regular. 
We may choose a morphism of models $f_A\colon X_A\to Y_A$ that is projective, birational, and with
$X_A$ smooth over $\Spec A$. We may also assume that $\Spec A$ is smooth over 
$\Spec \ZZ$. Since $\cO_Y\hookrightarrow f_*(\cO_X)$, we may assume that
$\cO_{Y_A}\hookrightarrow (f_A)_*(\cO_{X_A})$.  In particular, $Y_A$ is reduced.
Furthermore, by
generic base-change we may assume that $\cO_{Y_s}\hookrightarrow (f_s)_*(\cO_{X_s})$ for every
$s\in\Spec A$. In particular, $Y_s$ is reduced, and if $Y$ is irreducible, then so are all
$Y_s$ (here we make use of the assumption that $k$ is algebraically closed). We also see that $\dim(Y_s)=\dim(Y)$ for all $s$, since we know this property 
for $X$. Similarly, if $Y$ is normal, then $\cO_Y= f_*(\cO_X)$, and arguing as above we may assume that $Y_A$ and all $Y_s$ are normal.

If $D=a_1D_1+\ldots+a_rD_r$ is a Weil divisor on $Y$, then we may assume that we have 
prime divisors $(D_i)_A$ on $Y_A$, and let $D_A:=\sum_ia_i (D_i)_A$. After possibly replacing 
$A$ by a localization $A_a$, we may assume that for every $s\in \Spec\,A$ the fiber
$(D_i)_s$ is a prime divisor on $Y_s$, and we get the divisor $D_s=\sum_ia_i (D_i)_s$.

In particular, if $Y$ is irreducible and normal, we may consider $K_Y$, a Weil divisor unique up to linear equivalence, whose restriction to the nonsingular locus
$Y_{\rm sm}$ is a divisor corresponding to $\omega_{Y_{\rm sm}}$. We write $K_{Y_A}$
for $(K_Y)_A$. If $U=Y_{\rm sm}$, then we may assume that the corresponding open
subset $U_A\subset Y_A$ is smooth over $A$, and $K_{Y_A}$ is a divisor whose restriction
to $U_A$ corresponds to
$\Omega^n_{U_A/A}$, where $n=\dim(Y)$. We may therefore assume that for every $s\in\Spec A$, the restriction of $K_{Y_A}$ to $Y_s$ gives a canonical divisor $K_{Y_s}$.

\section{Test ideals and multiplier ideals}\label{test_and_multiplier}

\subsection{Multiplier ideals}\label{mult_ideals}
We start by recalling the definition of multiplier ideals. For details, basic properties, and
further results we refer to \cite{positivity}.
Let $k$ be an algebraically closed field of characteristic zero, and $Y$ an irreducible normal scheme of finite type over $k$. 
 We consider a Weil divisor
$D$ on $Y$ such that $K_Y+D$ is Cartier\footnote{One can assume that $D$ is just a $\QQ$-divisor such that some multiple of $K_Y+D$ is Cartier; however, we will not need this level of generality.}. Given a nonzero ideal $\fra$ on $Y$,
we define the multiplier ideals $\cJ(Y,D,\fra^{\lambda})$ for $\lambda\in\RR_{\geq 0}$, as follows.

Recall first that given any birational morphism $\pi\colon X\to Y$, with $X$ normal, there is a unique
divisor
$D_X$ on $X$ with the following two properties:
\begin{enumerate}
\item[i)] $K_X+D_X$ is linearly equivalent with $\pi^*(K_Y+D)$ (hence, in particular, it is
Cartier).
\item[ii)] For every non-exceptional prime divisor $T$ on $X$, its coefficient in $D_X$ 
is equal to
its coefficient in the strict transform $\widetilde{D}$ of $D$.
\end{enumerate}
Note that $D_X$ is supported on $\widetilde{D}+{\rm Exc}(\pi)$, where ${\rm Exc}(\pi)$ is the exceptional locus of $\pi$. 

Suppose now that $\pi\colon X\to Y$ is a log resolution of the triple $(Y,D,\fra)$. This means that $\pi$ is projective and birational, $X$ is nonsingular, $\fra\cdot\cO_X=\cO_X(-G)$ for a divisor $G$, 
${\rm Exc}(\pi)$ is a divisor, and $E:=\widetilde{D}+{\rm Exc}(\pi)+G$ has simple normal crossings.  
With this notation, we have
\begin{equation}\label{eq_mult}
\cJ(X,D,\fra^{\lambda}):=\pi_*\cO_X(-D_X-\lfloor\lambda\cdot G\rfloor).
\end{equation}
Recall that if $T=\sum_ib_iT_i$ is an $\RR$-divisor, then $\lfloor T\rfloor:=\sum_i\lfloor b_i\rfloor
T_i$, where $\lfloor b_i\rfloor$ is the largest integer $\leq b_i$.
When $\fra=(f)$ is a principal ideal, then we simply write $\cJ(X,D,f^{\lambda})$.
Note that $\cJ(X,D,\fra^{\lambda})$ is in general only a fractional ideal. However, if $D$ is effective,
then all components of $D_X$ with negative coefficient are exceptional. Therefore in this case
$\cJ(X,D,\fra^{\lambda})$ is an ideal. 

It is a basic fact that the above definition is independent of resolution. It follows from 
(\ref{eq_mult}) that $\cJ(Y,D,\fra^{\lambda})\subseteq\cJ(Y,D,\fra^{\mu})$ if $\lambda>\mu$. 
Furthermore, given any $\lambda\geq 0$, there is $\epsilon>0$ such that 
$\cJ(Y,D,\fra^{\lambda})=\cJ(Y,D,\fra^{\mu})$ for all $\mu$ with $\lambda\leq\mu \leq
\lambda+\epsilon$.  One says that $\lambda>0$ is a \emph{jumping number} of $(Y,D,\fra)$
if $\cJ(Y,D,\fra^{\lambda})\neq\cJ(Y,D,\fra^{\mu})$ for every $\mu<\lambda$. Note that if we write
$G=\sum_{i=1}^Nb_iE_i$,
then for every jumping number $\lambda$ we must have
\begin{equation} \label{candidate_jump}
\lambda b_i\in\ZZ\,\,\,\text{for some}\,\,i\leq N\,\,\text{with}\,\,b_i\neq 0
\end{equation}
(if $\lambda$ satisfies this property, we call it a \emph{candidate jumping number}).
 In particular, the set of jumping numbers of $(X,D,\fra)$ is a discrete subset of $\QQ_{>0}$. 

We now recall a few properties of multiplier ideals that will come up later. The following result is 
\cite[Theorem 9.2.33] {positivity}. The proof uses the definition of multiplier ideals and the independence of resolutions (while the statement therein requires the varieties to be nonsingular,
the same proof works in the general setting) .

\begin{proposition}\label{invar_birat_mult}
If $\pi\colon X\to Y$ is any projective, birational morphism, with $X$ normal, and if $\fra'=\fra\cdot\cO_{X}$, then for every $\lambda\in\RR_{\geq 0}$ we have
$$\cJ(Y,D,\fra^{\lambda})=\pi_*\cJ(X,D_{X},(\fra')^{\lambda}).$$
\end{proposition}

We now consider a finite surjective morphism $\mu\colon Y'\to Y$, with $Y'$ normal and 
irreducible, and put
$\fra'=\fra\cdot\cO_{Y'}$. In this case there is an open subset $U\subseteq Y$ such that 
$\codim(Y\smallsetminus U,Y)\geq 2$, and both $U$ and $V=\phi^{-1}(U)$ are nonsingular
(for example, one can take $U=Y_{\rm sm}\smallsetminus \mu(Y'\smallsetminus Y'_{\rm sm})$. 
In this case both $K_{V/U}$ and $\mu^*(D\vert_U)$ are well-defined divisors on $V$.
We denote by $D_{Y'}$ the unique Weil divisor on $Y'$ whose restriction to $V$
is $\mu^*(D\vert_U)-K_{V/U}$. Note that $K_{Y'}+D_{Y'}$ is linearly equivalent with $\mu^*(K_Y+D)$, hence
in particular it is
Cartier. For an integral scheme
$W$, we denote by $K(W)$ the function field of $W$.

\begin{proposition}\label{invar_finite_mult}
With the above notation, for every $\lambda\in\RR_{\geq 0}$ we have
$$\cJ(Y,D,\fra^{\lambda})=\mu_*\cJ(Y',D_{Y'}, (\fra')^{\lambda})\cap K(Y).$$
\end{proposition}

\begin{proof}
If both $Y$ and $Y'$ are nonsingular, then the result is \cite[Theorem 9.5.42]{positivity}. Note that
the result therein only requires $\mu$ to be generically finite. The singular case is an easy consequence: if $X\to Y$ is a resolution of singularities, and $X'\to X\times_YY'$ is a resolution of singularities of the irreducible component dominating $Y'$, we get a commutative diagram
\begin{equation}\label{diag2}
\xymatrix{
X' \ar[r]^g \ar[d]_{\pi'} & X \ar[d]^\pi \\
Y' \ar[r]^{\mu} & Y
}
\end{equation}
with $\pi$ and $\pi'$ projective and birational, and $g$ generically finite. Applying 
\cite[Theorem 9.5.42]{positivity} to $g$, and Proposition~\ref{invar_birat_mult} to $\pi$ and $\pi'$, we deduce the 
equality in the proposition.
\end{proof}

\begin{remark}\label{rem_invar_finite_mult}
Note that if the divisor $D$ in Proposition~\ref{invar_finite_mult} is effective, then
the proposition implies that
$$\cJ(Y,D,\fra^{\lambda})=\mu_*\cJ(Y',D_{Y'}, (\fra')^{\lambda})\cap\cO_Y.$$
\end{remark}

The following statement follows directly from the definition of multiplier ideals and the projection formula, see \cite[Proposition~9.2.31]{positivity}.

\begin{proposition}\label{add_divisor}
Let $(Y,D,\fra)$ be as above, and suppose that $D'$ is a Cartier divisor on $Y$. For every $\lambda\geq 0$
we have
$$\cJ(Y,D+D', \fra^{\lambda})=\cJ(Y,D, \fra^{\lambda})\cdot \cO_Y(-D').$$
\end{proposition}

The following result is \cite[Proposition 9.2.28]{positivity}. It is a consequence of Bertini's theorem.

\begin{proposition}\label{general_elements}
Suppose that $Y=\Spec R$ is affine, $\fra=(h_1,\ldots,h_m)$, and $d$ is a positive integer.
If $g_1,\ldots,g_d$ are general linear combinations of the $h_i$ with coefficients in $k$, and if
$g=\prod_{i=1}^dg_i$, then
$$\cJ(Y,D, \fra^{\lambda})=\cJ(Y,D, g^{\lambda/d})$$
for every $\lambda<d$. 
\end{proposition}

We end this subsection with a statement of Skoda's theorem for multiplier ideals on singular varieties. For a proof, see \cite[Corollary~1.4]{LLS}. Note, however, that in this paper we will
only need the case when $X$ is nonsingular.  A proof in this case can be found in \cite[\S~11.1.A]{positivity}.

\begin{proposition}\label{Skoda}
Let $(Y,D, \fra)$ be as above. If $\fra$ can be locally generated by $m$ sections, then
$$\cJ(Y,D, \fra^{\lambda})=\fra\cdot\cJ(Y,D, \fra^{\lambda-1})$$
for every $\lambda\geq m$. 
\end{proposition}

\subsection{Reduction mod $p$ of multiplier ideals}\label{reduction_mod_p_multiplier}
Suppose now that $Y$, $D$ and $\fra$ are as in \S\ref{mult_ideals}, and let us consider a model $Y_A$ of
$Y$ over a finitely generated $\ZZ$-subalgebra $A$ of $k$. We  follow the notation introduced in 
\S\ref{reduction_mod_p}.
We may assume that we have a Weil divisor
$D_A$ on $Y_A$ and a sheaf of ideals $\fra_A$ on $Y_A$ that give models for $D$ and $\fra$.
Let us fix now a log resolution $\pi\colon X\to Y$ of $(Y,D,\fra)$. We may assume that this 
is induced by a projective birational morphism $X_A\to Y_A$, with $X_A$ smooth over $A$. 

Note that since $K_Y+D$ is Cartier, we may assume that $K_{Y_A}+D_A$ is Cartier,
and all $K_{Y_s}+D_s$ are Cartier, for $s\in\Spec A$. The divisor $D_{X_A}$ on $X_A$
that induces $D_X$ satisfies analogous properties to   properties i) and ii) stated in
\S\ref{mult_ideals} for $D_X$. Furthermore, for every $s\in\Spec A$, the restriction $D_{X_s}$ of $D_{X_A}$ to $X_s$ 
satisfies
\begin{enumerate}
\item[i)] The Cartier divisors $K_{X_s}+D_{X_s}$ and $\pi_s^*(K_{Y_s}+D_s)$ are linearly equivalent.
\item[ii)] For every non-exceptional prime divisor $T$ on $X_s$, its coefficient in $D_{X_s}$ is equal to
its coefficient in the proper transform $\widetilde{D_s}$ of $D_s$. 
\end{enumerate}

We may assume that $\fra_A\cdot\cO_{X_A}=\cO_{X_A}(-G_A)$ for a divisor $G_A$ on $X_A$, and that 
${\rm Exc}(\pi_A)$ is a divisor.
Furthermore, we may assume that we have a divisor $E_A=\sum_{i=1}^N (E_i)_A$ on 
$X_A$ such that every intersection $(E_{i_1})_A\cap\ldots\cap (E_{i_{\ell}})_A$ is smooth over $A$,
and such that $G_A$, ${\rm Exc}(\pi_A)$ and $D_{X_A}$ are supported on ${\rm Supp}(E_A)$. We deduce that
for every $s\in\Spec A$, the induced morphism $\pi_s\colon X_s\to Y_s$ gives a log resolution
of $(Y_s, D_s,\fra_s)$. 

Suppose now that $m$ is a positive integer such that $\fra$ can be generated locally by
$m$ sections. Recall that in this case we have by Proposition~\ref{Skoda}
$\cJ(Y,D, \fra^{\lambda})=\fra\cdot\cJ(Y,D, \fra^{\lambda-1})$ for every $\lambda\geq m$.
This allows us to focus on exponents $<m$ in defining 
$\cJ(Y,D, \fra^{\lambda})_s$, and then extend the definition by putting
$\cJ(Y,D, \fra^{\lambda})_s=\fra_s\cdot \cJ(Y,D, \fra^{\lambda-1})_s$
for $\lambda\geq m$.

For $\lambda<m$, we have the ideal
$(\pi_A)_*\cO_{X_A}(-(D_X)_A-\lfloor\lambda G_A\rfloor)$
on $Y_A$ that gives a model of $\cJ(Y,D, \fra^{\lambda})$. By generic base-change, we may assume that for every $s\in\Spec A$, the induced ideal $\cJ(Y,D, \fra^{\lambda})_s$ is the ideal
$(\pi_s)_*\cO_{X_s}(-D_{X_s}-\lfloor \lambda G_s\rfloor)$.
Indeed, note that we only have to consider finitely many ideals, corresponding to the 
candidate jumping numbers as in (\ref{candidate_jump}) that are $<m$.
We mention that if we consider the above construction starting with a different log resolution of $(Y,D,\fra)$,
then
there is an open subset of $\Spec A$ such that for every $s$ in this subset, the two definitions of the ideals $\cJ(Y,D, \fra^{\lambda})_s$ coincide.

\subsection{Test ideals}\label{section_test_ideals}
In this subsection we work over a perfect field $L$ of positive characteristic $p$ (in the case of interest for us, $L$ will always be a finite field).
In this case, the test ideals of Hara and Yoshida \cite{HY} admit a simpler description, due to Schwede \cite{Schwede}, that completely avoids tight closure theory. Our main reference here is
\cite{ST}, though for some of the proofs the reader will need to consult the references given therein.

Before giving the definition of test ideals, we review a fundamental map in positive characteristic.
Suppose that $Y$ is a smooth connected scheme over $L$, of dimension $n$.
Let $\Omega^{\bullet}_{Y/L}$ be the de Rham complex on $Y$. If $F$ denotes the absolute Frobenius morphism on $Y$, then the \emph{Cartier isomorphism} is a graded
$\cO_Y$-linear isomorphism 
$C_Y\colon \oplus_i {\mathcal H}^i(F_*(\Omega^{\bullet}_{Y/L}))
\to\oplus_i\Omega^i_{Y/L}$
(see \cite{DI} for description and proof). In particular, we get a surjection
$$F_*\omega_Y=F_*\Omega_{Y/L}^n\to {\mathcal H}^n(F_*(\Omega^{\bullet}_{Y/L}))
\overset{C_Y}\to\omega_Y,$$
that we denote by $t_Y$. Iterating this map $e$ times gives $t_Y^e\colon F^e_*\omega_Y\to
\omega_Y$.

If $f$ and $w$ are local sections of $\cO_Y$ and $\omega_Y$,
respectively, then $t_Y\left(\frac{1}{f}w\right)=\frac{1}{f}t_Y(f^{p-1}w)$. This shows that for every
effective divisor $D$ on $Y$, we have an induced map
$$t_{Y,D}\colon F_*(\omega_Y(D))\to\omega_Y(D),$$ compatible with the previous one
via the inclusion $\omega_Y\hookrightarrow\omega_Y(D)$. The same remark applies to the
maps $t_Y^e$. If $D$ is not necessarily effective, then $t_{Y,D}$ is still well-defined, but its image lies in the sheaf $\omega_Y\otimes K(Y)$ of rational $n$-forms on $Y$.

The map $t_Y\colon F_*(\omega_Y)\to\omega_Y$ can be described around a closed point 
$y\in Y$, as follows.
Let us choose a system of coordinates $u_1,\ldots,u_n$ at $y$
(that is, a regular system of parameters of $\cO_{Y,y}$). We may assume that we have
an affine open neighborhood $U$ of $y$ such that $u_i\in\cO_Y(U)$ for all $i$, and that
$du=du_1\wedge\ldots\wedge du_n$ gives a basis of $\omega_Y\vert_U$. 
Furthermore, the residue field of $\cO_{Y,y}$ is finite over the perfect field $L$, hence it is perfect,
and since the $u_i$ give a regular system of parameters at $y$, 
we may assume that $\cO_Y(U)$ is free over $\cO_Y(U)^p$, with a basis given by
$$\{u_1^{i_1}\cdots u_n^{i_n}\mid 0\leq i_j\leq p-1\,\text{for}\,1\leq j\leq n\}.$$
In this case $t_Y\vert_U$ is characterized by the fact that $t_Y(h^p w)=h\cdot t_Y(w)$ for every
$h\in\cO_Y(U)$, and for every $i_j$ with $0\leq i_j\leq p-1$ we have
\begin{equation}\label{formula_t}
t_Y(u_1^{i_1}\cdots u_n^{i_n}du)=\left\{
\begin{array}{cl}
du, & \text{if}\,\,i_j=p-1\,\text{for all}\, \,j; \\[1mm]
0, & \text{otherwise}.
\end{array}\right.
\end{equation}

The map $t_Y$ is functorial in the following sense. Consider a morphism
$\pi\colon X\to Y$ of smooth schemes over $L$. For every $i$ we have a commutive diagram 
involving the respective Cartier isomorphisms

\[
\begin{CD}
\pi^*{\mathcal H}^i(F_*\Omega_{Y/L}^{\bullet})@>{\pi^*(C_Y)}>>
 \pi^*\Omega_{Y/L}^i\\
@V{\beta_i}VV @VV{\alpha_i}V \\
{\mathcal H}^i(F_*\Omega_{X/L}^{\bullet})@>{C_X}>> \Omega_{X/L}^i 
\end{CD}
\]
where $\alpha_i$ is given by pulling-back forms, while $\beta_i$ is obtained as the composition
$$\pi^*{\mathcal H}^i(F_*\Omega_{Y/L}^{\bullet})\to 
{\mathcal H}^i(\pi^*(F_*\Omega_{Y/L}^{\bullet}))   
\to {\mathcal H}^i(F_*\pi^*\Omega_{Y/L}^{\bullet})
\to {\mathcal H}^i(F_*\Omega_{X/L}^{\bullet}).$$
If, in addition, $\pi$ is a proper birational map, $D$ is an effective divisor on $Y$, and 
$D_X$ is the divisor on $X$ defined as in \S\ref{mult_ideals}, we get 
an induced commutative diagram relating the two trace maps
\begin{equation}\label{diag3_1}
\begin{CD}
\pi^*(F^e_*(\omega_Y(D)))@>{\pi^*(t_{Y,D}^e)}>>\pi^*(\omega_Y(D))\\
@VVV @VVV\\
F^e_*(\omega_X(D_X))@>{t_{X,D_X}^e}>>\omega_X(D_X)\otimes K(X)
\end{CD}
\end{equation}
where the right vertical map is obtained by composing the isomorphism
$\psi\colon\pi^*(\omega_Y(D))\to\omega_X(D_X)$ with the inclusion
$\omega_X(D_X)\hookrightarrow\omega_X(D_X)\otimes K(X)$, 
and the left vertical map is given by the composition 
\[
\begin{CD}
\pi^*(F^e_*(\omega_Y(D)))@>>> F^e_*(\pi^*(\omega_Y(D)))@>{F^e_*(\psi)}>> 
F^e_*(\omega_X(D_X)).
\end{CD}
\]
Note that while the botom horizontal map in (\ref{diag3_1}) does not necessarily
land in $\omega_X(D_X)$ (since in general $D_X$ is not effective), the composition of the maps
in (\ref{diag3_1}) has this property.

Suppose now that $Y$ is a normal, irreducible scheme over $L$, of dimension $n$. 
We fix an
effective  
Weil divisor $D$ on $Y$ such that $K_Y+D$ is Cartier
(note that in \cite{ST} one works in a more general framework, which complicates some of the definitions; for the sake of simplicity, we only give the definitions in the setting that we will need).
We claim that to $D$ and to every $e\geq 1$ we can naturally associate an $\cO_Y$-linear map
\begin{equation}\label{map_to_divisor}
\phi_{D}^{(e)}\colon F^e_*\cO_Y((1-p^e)(K_Y+D))\to\cO_Y.
\end{equation}
Indeed, in order to define $\phi_D^{(e)}$ it is enough to do it on the complement of a closed 
subset of codimension $\geq 2$. Therefore we may assume that $Y$ is smooth over $L$.
In this case $\phi_D^{(e)}$ is obtained by tensoring
$t_{Y,D}^e\colon F^e_*(\omega_Y(D))\to \omega_Y(D)$ by $\omega_Y^{-1}(-D)$, and using the 
projection formula for $F^e$. Note that if $D$ is not necessarily effective, then we may still
define as above $\phi_D^{(e)}$, but its image will be a fractional ideal on $Y$, not necessarily
contained in $\cO_Y$. 

If $\pi\colon X\to Y$ is a proper, birational morphism of schemes, with $X$ smooth, then we 
have as in \S\ref{mult_ideals} a unique divisor $D_X$ on $X$ such that $K_X+D_X$ is linearly equivalent to
$\pi^*(K_Y+D)$, and such that $D_X$ agrees along the non-exceptional divisors of $\pi$
with the proper transform of $D$. In this case, we claim that 
the commutative diagram (\ref{diag3_1})
induces a commutative diagram
\begin{equation}\label{diag3_2}
\begin{CD}
\pi^*F^e_*\cO_Y((1-p^e)(K_Y+D))@>{\pi^*(\phi_D^{(e)}})>>\pi^*\cO_Y\\
@VVV @VVV\\
F^e_*\cO_X((1-p^e)(K_X+D_X))@>{\phi_{D_X}^{(e)}}>>K(X),
\end{CD}
\end{equation}
where the right vertical map is given by $\pi^*(\cO_Y)\simeq\cO_X\hookrightarrow K(X)$.
This follows when $Y$ is smooth, too, using the commutativity of
(\ref{diag3_1}). In the general case, note that
(\ref{diag3_2}) corresponds by the adjointness of $(\pi^*,\pi_*)$ to the diagram
\begin{equation}\label{diag3_3}
\begin{CD}
F^e_*\cO_Y((1-p^e)(K_Y+D))@>\phi_D^{(e)}>>\cO_Y\\
@VF^e_*({\rho})VV @VVV\\
\pi_*F^e_*\cO_X((1-p^e)(K_X+D_X))@>{\pi_*(\phi_{D_X}^{(e)}})>>K(Y),
\end{CD}
\end{equation}
where $\rho\colon\cO_Y((1-p^e)(K_Y+D))\to\pi_*\cO_X((1-p^e)(K_X+D_X))$
is the canonical isomorphism given by pull-back of sections.
In order to check the commutativity of (\ref{diag3_3}) we may restrict to the complement of a codimension $\geq 2$ closed subset, and therefore we may assume that both $X$ and $Y$ are smooth, in which case, as we have mentioned, (\ref{diag3_2}) hence also (\ref{diag3_3}) is 
commutative. Note also that since (\ref{diag3_3}) is commutative and the left vertical map is an isomorphism, the image of $\pi_*(\phi_{D_X}^{(e)})$ is contained in $\cO_Y=\pi_*(\cO_X)$.

\begin{example}\label{computation_SNC}
Suppose that $Y$ is nonsingular, and $D=a_1E_1+\ldots+a_rE_r$ is a not-necessarily-effective
 simple normal crossings divisor on $Y$. If $\frb=\cO_Y(-D')$, where $D'=\sum_{i=1}^rb_iE_i$
 is effective, then
 $\phi_D^{(e)}(F^e_*(\frb\cdot\cO_Y((1-p^e)(K_Y+D))))$ is the fractional ideal $\cO_Y(-D-F)$,
 where $F=\sum_{i=1}^rc_iE_i$, with $c_i=\lfloor (b_i-a_i)/p^e\rfloor$ for every $i$.
 This description follows easily from the description in coordinates of the map $t_{Y,D}^e$. 
\end{example}

We can now recall the definition of the test ideal $\tau(Y,D, \fra^{\lambda})$, 
where $(Y,D)$ is as above (with $D$ effective), $\fra$ is a 
nonzero ideal on $Y$, and $\lambda$ is a non-negative real number. One shows that there is a unique
minimal nonzero coherent ideal sheaf $J$ on $Y$ such that for every $e$ we have
\begin{equation}\label{eq_def_test}
\phi_D^{(e)}\left(F^e_*(\fra^{\lceil \lambda(p^e-1)\rceil} J\cdot \cO_Y((1-p^e)(K_Y+D)))\right)
\subseteq J.
\end{equation}
This is the test ideal $\tau(Y,D, \fra^{\lambda})$.
Here $\lceil u\rceil$ denotes the smallest integer $\geq u$. When $\fra=(f)$ is a principal ideal,
we simply write $\tau(Y,D,f^{\lambda})$.

\begin{proposition}\label{prop_power}
If $(Y,D)$ is as above, and $d$ is a positive integer, then
$$\tau(Y,D, (\fra^d)^{\lambda})\subseteq\tau(Y,D, \fra^{d\lambda})\,\,\,\text{for every}\,\,\,\lambda\in\RR_{\geq 0}.$$
\end{proposition}

\begin{proof}
For every $e$ we have $d\lceil \lambda(p^e-1)\rceil \geq \lceil d\lambda (p^e-1)\rceil$. It follows that
if $J$ satisfies (\ref{eq_def_test}) with $\fra^{\lceil \lambda(p^e-1)\rceil}$ replaced by
$\fra^{\lceil d\lambda(p^e-1)\rceil}$, then it also satisfies (\ref{eq_def_test}) with 
$\fra^{\lceil \lambda(p^e-1)\rceil}$ replaced by
$\fra^{d \lceil \lambda(p^e-1)\rceil}$. The assertion in the proposition now follows from
the minimality in the definition of $\tau(Y,D, (\fra^d)^{\lambda})$. 
\end{proof}

\begin{remark}
In fact, the inclusion in proposition~\ref{prop_power} is an equality. We leave to the interested reader the task of checking the reverse inclusion, that we will not need. 
See \cite[Corollary~2.15]{BMS} for a proof in the case when $Y$ is nonsingular.
\end{remark}

In order to describe $\tau(Y,D, \fra^{\lambda})$, it is enough to do it when $Y=\Spec R$ is affine.
In this case one can show (see \cite[Lemma~6.4]{ST}) that there is a nonzero $c\in R$ such that
for every nonzero $g\in R$, there is $e\geq 1$ such that
\begin{equation}\label{eq0_def_test}
c\in \phi_D^{(e)}\left(F^e_*(\fra^{\lceil \lambda(p^e-1)\rceil} g\cdot \cO_Y((1-p^e)(K_Y+D)))\right).
\end{equation}
In this case, it is not hard to see that 
\begin{equation}\label{eq2_def_test}
\tau(Y,D, \fra^{\lambda})=\sum_{e\geq 1}\phi_D^{(e)}
\left(F^e_*(\fra^{\lceil \lambda(p^e-1)\rceil} c\cdot \cO_Y((1-p^e)(K_Y+D)))\right)
\end{equation}
(see \cite[Proposition~6.8]{ST}). 
For example, if $u\in \fra\smallsetminus\{0\}$ is such that $U=\Spec R_u$ is regular, and
$D\vert_U=0$, then one can take $c$ to be a power of $u$ (see
\cite[Remark~6.6]{ST}).

Note that if $\fra\subseteq\frb$ and $c\in R$ satisfies 
(\ref{eq0_def_test}) for $\fra$, then it also satisfies it for $\frb$. An immediate consequence of
the description (\ref{eq2_def_test}) for the test ideal is the following monotonicity property.

\begin{proposition}\label{monotonicity}
If $(Y,D)$ is as above, and $\fra$, $\frb$ are nonzero ideals on $Y$ with $\fra\subseteq\frb$, then
$\tau(Y,D, \fra^{\lambda})\subseteq\tau(Y,D, \frb^{\lambda})$ for every $\lambda\in\RR_{\geq 0}$.
\end{proposition}

The above gives a definition of $\tau(Y,D, \fra^{\lambda})$ in the case when 
$D$ is an effective divisor. On the other hand, one shows (see \cite[Lemma~6.11]{ST})
that if $D'$ is any effective Cartier divisor, then 
\begin{equation}\label{eq3_def_test}
\tau(Y,D+D', \fra^{\lambda})=\tau(Y,D, \fra^{\lambda})\cdot\cO_Y(-D'). 
\end{equation}
If $D$ is a not-necessarily-effective Weil divisor such that $K_Y+D$ is Cartier, then
we define $\tau(Y,D, \fra^{\lambda})$ as follows. Working locally, we can find a Cartier
divisor $D'$ such that $D+D'$ is effective, and in this case $\tau(Y,D, \fra^{\lambda})$ is the fractional ideal
$\tau(Y,D+D', \fra^{\lambda})\cdot\cO_Y(D')$. It follows from (\ref{eq3_def_test}) that the definition
is independent of the choice of $D'$. 

If $Y$ is nonsingular, one can show that the above definition for the test ideal
$\tau(Y,D,\fra^{\lambda})$ coincides with the one in \cite{BMS}, which is the one we gave in the
Introduction. We refer to \cite[Proposition~3.10]{BSTZ} for a proof.

While we will not need the results on the jumping numbers for the test ideals, we mention them because of the analogy with the case of multiplier ideals. For the proofs, see \cite{BMS} for the case when $Y$ is smooth and $D=0$, and \cite{BSTZ} for the general case. Given any $(Y,D, \fra)$ as above,
and any $\lambda\geq 0$, there is $\epsilon>0$ such that
$\tau(Y,D, \fra^{\lambda})=\tau(Y,D, \fra^{\mu})$ for every $\mu$ with $\lambda\leq\mu\leq\lambda+
\epsilon$. A positive $\lambda$ is an \emph{F-jumping number} if $\tau(Y,D, \fra^{\lambda})
\neq\tau(Y,D, \fra^{\mu})$ for every $\mu<\lambda$. One can show that the set of $F$-jumping numbers is a discrete set of rational numbers. However, we emphasize that this result is much more subtle than the corresponding one for multiplier ideals.

The following proposition gives the analogue of Skoda's theorem for test ideals 
(see \cite[Lemma~3.26]{BSTZ}). For the smooth case, which is the only one that we will
need in this paper, see \cite[Proposition~2.25]{BMS}.

\begin{proposition}\label{Skoda_test}
Let $(Y,D, \fra)$ be a triple as above, and $m$ a positive integer such that
$\fra$ is locally generated by $m$ sections. For every $\lambda\geq m$ we have
$$\tau(Y,D, \fra^{\lambda})=\fra\cdot \tau(Y,D, \fra^{\lambda-1}).$$
\end{proposition}

We will make use in \S\ref{connection_conjectures} 
of the following result of Schwede and Tucker
\cite[Corollary~6.28]{ST}
concerning the behavior of test ideals under finite morphisms. Let 
$\mu\colon Y'\to Y$ be a finite surjective morphism of normal, irreducible varieties. 
Given the Weil divisor $D$ on $Y$ such that $K_Y+D$ is Cartier, then as in 
Proposition~\ref{invar_finite_mult} we have a divisor $D_{Y'}$ on $Y'$ such that
$K_{Y'}+D_{Y'}$ and $\mu^*(K_Y+D)$ are linearly equivalent.
We also put $\fra'=\fra\cdot\cO_{Y'}$.

\begin{theorem}\label{finite_test}
With the above notation, if $\mu$ is a separable morphism  and if the trace map 
${\rm Tr}\colon K(Y')\to K(Y)$ is surjective, then 
$$\tau(Y,D, \fra^{\lambda})=\mu_*\tau(Y',D_{Y'}, (\fra')^{\lambda})\cap K(Y).$$
Furthermore, if $D$ is effective, then 
$$\tau(Y,D, \fra^{\lambda})=\mu_*\tau(Y',D_{Y'}, (\fra')^{\lambda})\cap \cO_Y.$$
\end{theorem}

One can compare this result with the corresponding result about multiplier ideals in 
Proposition~\ref{invar_finite_mult}. Note that the hypothesis in Theorem~\ref{finite_test}
is satisfied if $p={\rm char}(L)$ does not divide $[K(Y') : K(Y)]$. 

\section{The conjectural connection between multiplier ideals and test ideals}\label{conjectural_connection}

The following is the main conjecture relating multiplier ideals and test ideals.

\begin{conjecture}\label{conj_multiplier}
Let $Y$ be a normal, irreducible scheme over an algebraically closed field $k$ of characteristic zero.
Suppose that $D$ is a Weil divisor on $Y$ such that $K_Y+D$ is Cartier, and  $\fra$ is a nonzero ideal on $Y$. Given a model $Y_A$ of $Y$ over a ring $A\subset k$ of finite type over $\ZZ$,
there is a dense set of closed points $S\subset\Spec A$ such that 
\begin{equation}\label{eq_conj_multiplier}
\tau(Y_s,D_s, \fra_s^{\lambda})=\cJ(Y,D, \fra^{\lambda})_s
\,\,\text{for all}\,\,\lambda\in\RR_{\geq 0}\,\,\text{and all}\,\, s\in S. 
\end{equation}
Furthermore, if we have finitely many triples as above
$(Y^{(i)},D^{(i)},\fra^{(i)})$, and corresponding models over $\Spec A$, then there is
a dense open subset of closed points in $\Spec A$ such that (\ref{eq_conj_multiplier})
holds for each of these triples.
\end{conjecture}

One can formulate variants the above conjecture in more general settings. For example, $D$ may be assumed to be a $\QQ$-divisor such that some multiple of $K_Y+D$ is Cartier, and 
one can replace the ideal $\fra$ by finitely many ideals
$\fra_1,\ldots,\fra_r$. In the latter case one has to consider the corresponding \emph{mixed} multiplier and test ideals. On the other hand, in our main result we will restrict ourselves to the case when $X$
is nonsingular. In particular, in this case $D$ is Cartier, and therefore (\ref{eq_conj_multiplier})
holds if and only if it holds when $D=0$. Therefore in this case 
Conjecture~\ref{conj_multiplier} reduces to Conjecture~\ref{conj_introd2} in the Introduction.
For examples related to the above conjecture in the case of an ambient nonsingular variety, see
\cite[\S 3]{Mustata} and \cite[\S 4]{MTW}.

The inclusion ``$\subseteq$" in (\ref{eq_conj_multiplier}) is due to Hara and Yoshida
\cite{HY}. In fact, this inclusion holds for an open subset of closed points in $\Spec A$.
It is a consequence of the more precise result below.
We include a proof, since this is particularly easy with the alternative definition 
of test ideals that we are using.

\begin{proposition}\label{one_inclusion}
Let $Y$ be a normal irreducible scheme over a perfect field $L$ of positive characteristic $p$.
Suppose that $D$ is a divisor on $Y$ such that $K_Y+D$ is Cartier, 
and  $\fra$ is a nonzero ideal on $Y$. If $\pi\colon X\to Y$
is a proper birational morphism, with $X$ nonsingular, 
$\fra\cdot\cO_X=\cO_X(-G)$ for a divisor $G$, and ${\rm Supp}(G)\cup {\rm Supp}(D_X)$
has simple normal crossings, where the divisor  $D_X$ on $X$ is
defined as in \S\ref{mult_ideals}, then
\begin{equation}\label{eq_one_inclusion}
\tau(Y,D, \fra^{\lambda})\subseteq\pi_*\cO_X(-D_X-\lfloor \lambda G\rfloor)
\end{equation}
for every $\lambda\in\RR_{\geq 0}$.
\end{proposition}

\begin{proof}
After replacing $Y$ by each of the elements of a suitable open cover of $Y$, we may assume
that there is a Cartier divisor $D'$ on $Y$ such that $D+D'$ is effective. Since
it is enough to prove (\ref{eq_one_inclusion}) with $D$ replaced by $D+D'$, we may assume that
$D$ is effective.

Let $J$ denote the right-hand side of (\ref{eq_one_inclusion}). It follows from the minimality in the definition of the test ideal that in order to prove the inclusion in (\ref{eq_one_inclusion}),
it is enough to show that for every $e\geq 1$ we have the inclusion in
(\ref{eq_def_test}). Let us fix such $e$. Since $X$ is nonsingular and ${\rm Supp}(G)\cup {\rm Supp}(D_X)$ has simple normal crossings, if $\frb=\cO_X(-G)$, then $J':=\tau(X,D_X, \frb^{\lambda})
=\cO_X(-D_X-\lfloor\lambda G\rfloor)$. Indeed, this is is an easy consequence of the formula (\ref{eq2_def_test})
and of the one in Example~\ref{computation_SNC}
(note that in this case the
$c$ in (\ref{eq2_def_test}) can be taken to be a power of the defining equation of ${\rm Supp}(D_X)\cup {\rm Supp}(F)$).

By definition, we have 
\begin{equation}\label{eq2_one_inclusion}
\phi_{D_X}^{(e)}(F_*^e(\frb^{\ell} J'\cdot \pi^*(\cL)))\subseteq
J',
\end{equation}
where $\cL=\cO_Y((1-p^e)(K_Y+D))$ and $\ell=\lceil\lambda(p^e-1)\rceil$.
We now use the commutativity of (\ref{diag3_3}). With the notation therein we have
$\rho(\fra^{\ell}J\cdot \cL)\subseteq
\pi_*(\frb^{\ell}J'\cdot \pi^*(\cL))$. Therefore
$$\phi_D^{(e)}(F_*^e(\fra^{\ell}J\cdot \cL))=
\pi_*(\phi_{D_X}^{(e)})\left(F^e_*(\rho)(F^e_*(\fra^{\ell}J\cdot \cL))\right)
\subseteq\pi_*(\phi_{D_X}^{(e)})\left(F^e_*(\rho(\fra^{\ell}J\cdot \cL))\right)$$
$$\subseteq\pi_*\left(\phi_{D_X}^{(e)}(F^e_*(\frb^{\ell}J'\cdot\pi^*(\cL)))\right)
\subseteq\pi_*(J')=J,$$
where the last inclusion follows by applying $\pi_*$ to (\ref{eq2_one_inclusion}).
Therefore we have the inclusion in (\ref{eq_def_test}) for $J$, and this completes the proof
of the proposition.
\end{proof}

Note that in the setting of the conjecture, it is known that if we fix $\lambda$, then
we get the equality in (\ref{eq_conj_multiplier}) for all closed points in an open subset 
of $\Spec A$ (depending on $\lambda$). This was proved by Hara and Yoshida in \cite{HY}, relying on ideas that had been used also in \cite{Hara} and \cite{MS} \footnote{The result in \cite{HY} only treats the case of a local ring,
since one uses the tight closure approach to test ideals. However, one can modify the proof therein 
to give the assertion in our setting.}.

We end this section with the following proposition, that allows us to only consider
Conjecture~\ref{conj_introd2} in the case of principal ideals on nonsingular affine varieties.

\begin{proposition}\label{prop_reduction_principal}
In order to prove Conjecture~\ref{conj_introd2}, it is enough to consider the case when
$Y$ is an affine nonsingular variety and $\fra=(f)$ is a principal ideal
${\rm (}$but allowing several such 
pairs${\rm )}$.
\end{proposition}

\begin{proof}
Since every $Y$ admits a finite affine open cover $Y=\bigcup_iU_i$, and since proving the conjecture
for $(Y, \fra)$ is equivalent to proving it (simultaneously) for all $(U_i,\fra\vert_{U_i})$,
it follows that it is enough to consider the case when for all pairs we treat, the ambient scheme $Y$ is affine and nonsingular.

For such a pair $(Y,\fra)$,
let $h_1,\ldots,h_m$ be generators of $\fra$.
It follows from Proposition~\ref{one_inclusion} that we only need to guarantee the inclusion
\begin{equation}\label{eq_enough}
\cJ(Y, \fra^{\lambda})_s\subseteq\tau(Y_s,\fra_s^{\lambda}).
\end{equation}
Furthermore, in light of Propositions~\ref{Skoda} and \ref{Skoda_test} it is enough to only consider
the case $\lambda<m$. 

Let $g_1,\ldots,g_m$ be general linear combinations of the $h_i$ with coefficients in $k$,
and $g=\prod_{i=1}^m g_i$, so that by Proposition~\ref{general_elements} we have 
$\cJ(Y,\fra^{\lambda})=\cJ(Y,g^{\lambda/m})$ for all $\lambda< m$. 
As we have seen in \S\ref{reduction_mod_p_multiplier}, 
in the case of multiplier ideals of bounded exponents 
we only have to consider finitely many such exponents (the candidate jumping numbers), hence we may assume after taking a model over $A$ that for
every closed point $s\in\Spec A$ we have
\begin{equation}\label{eq_reduction_principal2} 
\cJ(Y,\fra^{\lambda})_s
=\cJ(Y,g^{\lambda/m})_s
\end{equation}
for all $\lambda<m$. 

Suppose now that we can find a dense set $S$ of closed points in $\Spec A$ such that
\begin{equation}\label{eq_reduction_principal1}
\cJ(Y,g^{\lambda/m})_s\subseteq \tau(Y_s,g_s^{\lambda/m})
\end{equation} 
for every $s\in S$ and every $\lambda<m$.
Since $g\in\fra^m$, we have
by Propositions~\ref{monotonicity}  and \ref{prop_power}
\begin{equation}\label{eq_reduction_principal3}
\tau(Y_s,g_s^{\lambda/m})\subseteq\tau(Y_s,(\fra_s^m)^{\lambda/m})
\subseteq\tau(Y_s,\fra_s^{\lambda})
\end{equation}
for every $s\in S$. Putting together (\ref{eq_reduction_principal2}), (\ref{eq_reduction_principal1}),
and (\ref{eq_reduction_principal3}), we get (\ref{eq_enough}), which completes the proof of the proposition.
\end{proof}

\section{A conjecture regarding the Frobenius action on the cohomology of the structure sheaf}\label{conjecture_Frobenius_action}

In this section we discuss our conjecture about Frobenius actions, and deduce some consequences. Let $k$ be an algebraically closed field of characteristic zero.
We will freely use the notation and notions introduced in \S\ref{p_linear}
 and \S\ref{reduction_mod_p}. 
Recall the conjecture made in the Introduction:
suppose that
 $X$ is a connected, nonsingular $n$-dimensional projective algebraic variety over $k$,
and $X_A$ is a model of $X$ over the finitely generated $\ZZ$-subalgebra $A$ of $k$.
Conjecture~\ref{conj_introd} asserts that
there is
a dense set of closed points $S\subset\Spec A$ such that the Frobenius action
$F\colon H^n(X_s,\cO_{X_s})\to H^n(X_s,\cO_{X_s})$ 
is semisimple for every $s\in S$.

\begin{remark}\label{rem_ordinary1}
In fact, one expects that the analogous assertion would be true for the Frobenius action on each of the
cohomology vector spaces $H^i(X_s,\cO_{X_s})$. Moreover, it is expected that there is a dense set
of closed points $s\in\Spec A$ such that each $X_s$ is ordinary in the sense of Bloch and Kato \cite{BK} (see also \cite[Expos\'{e} III]{CL} for a nice introduction to ordinary varieties). As follows from
\cite[Proposition 7.3]{BK}, if $X_s$ is ordinary, then the action of Frobenius on the Witt vector cohomology
$H^i(X_s, W\cO_{X_s})$ is bijective. Note that we have an exact sequence of sheaves of abelian groups
$$0\to W\cO_{X_s}\overset{V}\to W\cO_{X_s}\to\cO_{X_s}\to 0$$
that is compatible with the action of Frobenius, 
where $V$ is the \emph{Verschiebung} operator. From the long exact
sequence 
$$H^i(X_s,W\cO_{X_s})\to H^i(X_s, W\cO_{X_s})\to H^i(X_s,\cO_{X_s})\to 
H^{i+1}(X_s,W\cO_{X_s})\to H^{i+1}(X_s,W\cO_{X_s})$$
that is compatible with the action of Frobenius, and the 
5-Lemma, it follows that Frobenius acts bijectively on $H^i(X_s,\cO_{X_s})$.

However, our hope is that proving that the Frobenius action on $H^n(X_s,\cO_{X_s})$ is semisimple would be easier than showing that $X_s$ is ordinary.
\end{remark}

\begin{remark}\label{rem_ordinary2}
If Conjecture~\ref{conj_introd} holds, then given finitely many varieties $X^{(1)},\ldots,X^{(r)}$ as above,
with $\dim(X^{(i)})=d_i$,
we may consider models $X^{(1)}_A,\ldots,X^{(r)}_A$ over $A$. In this case, there is a dense set of
closed points $S\subset\Spec A$ such that the action of $F$ on each cohomology group $H^{d_i}(X^{(i)}_s,
\cO_{X^{(i)}_s})$, with $s\in S$, 
is semisimple. Indeed, it is enough to apply the conjecture for $X=X^{(1)}\times\cdots\times X^{(r)}$,
using Remark~\ref{rem_Frob_product}
and the fact that by K\"{u}nneth's Formula we have
$$H^d(X_s,\cO_{X_s})={\bigotimes_{i=1}^r}H^{d_i}(X_s^{(i)},\cO_{X_s^{(i)}}),$$
where $d=\dim(X)=\sum_{i=1}^rd_i$.
\end{remark}

\begin{proposition}\label{reduction_to_number_field}
In order to prove Conjecture~\ref{conj_introd} for every algebraically closed field $k$
of characteristic zero, it is enough to prove it for the field of algebraic numbers $k=\overline{\QQ}$.
\end{proposition}

\begin{proof}
Suppose that $X$ is defined over $k$, and let $X_A$ be a model over $A$, where $A\subset k$
is a $\ZZ$-algebra of finite type. As pointed out in Remark~\ref{enough_nonempty}, it is enough
to show that there is a closed point $s\in\Spec A$ such that the Frobenius action on 
$H^n(X_s,\cO_{X_s})$ is semisimple. 

The $\QQ$-algebra $A_{\QQ}:=A\otimes_{\ZZ}{\QQ}$
is finitely generated, hence if $\frm$ is a prime ideal of $A$ such that $\frm A_{\QQ}$
is a maximal ideal of $A_{\QQ}$, then $K=A_{\QQ}/\frm A_{\QQ}$ is a finite extension
of $\QQ$. If $O_K$ is the ring of integers in $K$, then using the finite generation of $A$ over $\ZZ$
we see that there is a nonzero $h\in O_K$ such that the surjective morphism $A_{\QQ}\to K$
induces a morphism $A\to B=(O_K)_h$. 

Let $X_B=X_A\times_{\Spec A}\Spec B$ and 
$X_{\overline{\QQ}}=X_A\times_{\Spec A}\Spec\overline{\QQ}$,
where the morphism $A\to\overline{\QQ}$ is given by the composition $A\to K\hookrightarrow
\overline{\QQ}$. Since we may assume that $X_A$ is smooth and projective over 
$\Spec A$, with geometrically connected generic fiber,
it follows that $X_{\overline{\QQ}}$ is connected, smooth and projective over $\overline{\QQ}$, and clearly 
$X_B$ is a model of $X_{\overline{\QQ}}$ over $\Spec B$. If we know 
Conjecture~\ref{conj_introd} over $\overline{\QQ}$, then it follows that there is a closed point
$t\in \Spec B$ such that the Frobenius action on $H^n((X_{\overline{\QQ}})_t,
\cO_{(X_{\overline{\QQ}})_t})$ is semisimple.
If $s\in\Spec A$ is the image of $t$, then we have a finite field extension 
$k(s)\hookrightarrow k(t)$, and 
$(X_{\overline{\QQ}})_t=X_s\times_{\Spec k(s)}\Spec k(t)$.
Since $H^n((X_{\overline{\QQ}})_t,
\cO_{(X_{\overline{\QQ}})_t})\simeq H^n(X_s,\cO_{X_s})\otimes_{k(s)}k(t)$,
we conclude that the Frobenius action on $H^n(X_s,\cO_{X_s})$
is semisimple.
\end{proof}

\begin{example}\label{ex_abelian_variety}
If $X$ is a $g$-dimensional abelian variety, then we may assume that $X_s$ is an abelian variety over $k(s)$
for every closed point  $s\in \Spec A$. In this case $h^1(X_s,\cO_{X_s})=g$, and the action of Frobenius on 
$H^g(X_s,\cO_{X_s})\simeq\wedge^gH^1(X_s,\cO_{X_s})$ is semisimple if and only if 
the action of Frobenius on $H^1(X_s,\cO_{X_s})$ is semisimple. This is the case if and only if
$X_s$ is ordinary in the usual sense, that is, if $X_s\times_{\Spec k(s)}\Spec\overline{k(s)}$
has $p^g$ $p$-torsion points, where $p={\rm char}(k(s))$.

By Proposition~\ref{reduction_to_number_field}, in order to check Conjecture~\ref{conj_introd}
in this case we may assume that $X$ is defined over $\overline{\QQ}$. The conjecture is then known 
if $g\leq 2$, but it is open in general. The case of elliptic curves is classical, while the case
$g=2$ is due to Ogus \cite[Proposition 2.7]{Ogus} (see also \cite[Th\'{e}or\`{e}me 6.3]{CL} for a proof of this result).
\end{example}

\begin{example}\label{ex_curves}
If $X$ is a smooth projective 
curve of genus $g$, then the action of Frobenius on $H^1(X_s,\cO_{X_s})$
is semisimple if and only if the Jacobian of $X_s$ is ordinary in the usual sense. As pointed out in the previous example, Conjecture~\ref{conj_introd} is known in this case for $g\leq 2$, but it is open 
even in this case for $g\geq 3$.
\end{example}

In what follows we will assume Conjecture~\ref{conj_introd} (for all smooth, connected projective varieties), and then deduce several stronger
versions, working in the relative setting, and in the presence of a simple normal crossings divisor.
We start by considering a pair $(X,E)$, where $X$ is a connected, nonsingular $n$-dimensional projective variety
over $k$, and $E=E_1+\ldots+E_r$ is a reduced simple normal crossings divisor on $X$. 
Let $X_A$ be a model of $X$ over $\Spec A$. We may assume that $X_A$ is smooth over $A$,
and that we have irreducible divisors $(E_i)_A$ on $X_A$ giving models for the $E_i$, such that
every intersection $(E_{i_1})_A\cap\ldots\cap (E_{i_m})_A$ is smooth over $A$. In particular, 
if we put $E_A=\sum_{i=1}^r (E_i)_A$, then
for every
closed point $s\in\Spec A$, the divisor $E_s$ on $X_s$ has simple normal crossings.

\begin{lemma}\label{lem_SNC}
With the above notation, if Conjecture~\ref{conj_introd} holds, then
there is a dense set of closed points $S\subset\Spec A$ such that
 the Frobenius action
$F\colon H^{n-1}(E_s,\cO_{E_s})\to H^{n-1}(E_s,\cO_{E_s})$ is semisimple for all $s\in S$.
\end{lemma}

\begin{proof}
Let us fix a closed point $s\in\Spec A$. 
For every subset $J\subseteq\{1,\ldots,r\}$ we put $(E_J)_s=\bigcap_{i\in J}(E_i)_s$ (of course, these sets will be empty for some $J$). Note that we have an acyclic complex
$$C^{\bullet}:\,\,\,\,\,\,\,\,
0\to C^0\overset{d^0}\to C^1\overset{d^1}\to \cdots\overset{d^{n-1}}\to C^n\to 0,$$
where $C^0=(\cO_{E})_s$, and for all $p>0$ we have
$C^p=\oplus_{|J|=p}\cO_{(E_J)_s}$. Furthermore, we have a morphism of complexes
$C^{\bullet}\to F_*C^{\bullet}$. If we put $Z^i={\rm Ker}(d^i)$ for $1\leq i\leq n-1$
and $Z^n=C^n$, then we have exact sequences
\begin{equation}\label{exact_sequence1}
H^{n-p-1}(X_s, Z^{p+1})\to H^{n-p}(X_s, Z^p)\to H^{n-p}(X_s, C^p)
\end{equation}
compatible with the action of  Frobenius. Applying Conjecture~\ref{conj_introd} to all
connected components of all the intersections $E_{i_1}\cap\ldots\cap E_{i_m}$
simultaneously (see Remark~\ref{rem_ordinary2}), we see that we have a dense set of closed points 
$S\subset \Spec A$ such that the Frobenius action on each $H^{n-p}(X_s, C^p)$ is semisimple for
$p\geq 1$ and $s\in S$. Using Lemma~\ref{rem_Frob_exact}
and the exact sequences (\ref{exact_sequence1}), we see by descending induction on
$p\leq n$ that for every $s\in S$, the Frobenius action on each $H^{n-p}(X_s,Z^p)$
is semisimple. By taking $p=1$, we get the assertion in the lemma.
\end{proof}

\begin{corollary}\label{cor_lem_SNC}
With the notation in the lemma, and still assuming Conjecture~\ref{conj_introd}, there is a dense set of closed points $S\subset\Spec A$ such that
the Frobenius action 
$$F\colon H^n(X_s,\cO(-E_s))\to H^n(X_s,\cO(-E_s))$$ is semisimple for every $s\in S$.
\end{corollary}

\begin{proof}
Consider the exact sequence
$$T^{\bullet}:\,\,\,\,0\to\cO_{X_s}(-E_s)\to\cO_{X_s}\to\cO_{E_s}\to 0.$$
Note that we have a morphism of exact sequences $T^{\bullet}\to F_*T^{\bullet}$.
Applying Lemma~\ref{rem_Frob_exact} to the exact sequence
$$H^{n-1}(E_s,\cO_{E_s})\to H^n(X_s,\cO_{X_s}(-E_s))\to H^n(X_s,\cO_{X_s}),$$
as well as Conjecture~\ref{conj_introd} to $H^n(X_s,\cO_{X_s})$ and Lemma~\ref{lem_SNC} to
$H^{n-1}(E_s,\cO_{E_s})$ (note that we can apply these simultaneously by 
Remark~\ref{rem_ordinary2}), it follows that the Frobenius action on $H^n(X_s,\cO(-E_s))$
is semisimple for all $s$ in a suitable dense set of closed points $S\subset\Spec A$.
\end{proof}

Still keeping the above notation, let $s\in\Spec A$ be a closed point. Recall that we have a canonical surjective $\cO_{X_s}$-linear map $t_s:=t_{X_s,E_s} \colon F_*(\omega_{X_s}(E_s))\to \omega_{X_s}(E_s)$ induced by the Cartier isomorphism. For every $e\geq 1$ we also consider the composition 
$t_s^e\colon F^e_*(\omega_{X_s}(E_s))\to \omega_{X_s}(E_s)$.

\begin{corollary}\label{cor2_lem_SNC}
With the notation in Lemma~\ref{lem_SNC}, and assuming that Conjecture~\ref{conj_introd} holds, there is a dense set of closed points
$S\subset \Spec A$ such that the map induced by $t_s^e$
$$H^0(X_s,F^e_*(\omega_{X_s}(E_s)))\to H^0(X_s,\omega_{X_s}(E_s))$$
is surjective for all $e\geq 1$ and all $s\in S$. 
\end{corollary}

\begin{proof}
It is enough to show that every closed point $s\in\Spec A$ that satisfies Corollary~\ref{cor_lem_SNC} also satisfies our conclusion. As abelian groups, we have
$H^0(X_s,F^e_*(\omega_{X_s}(E_s)))=H^0(X_s,\omega_{X_s}(E_s))$, and the map induced
by $t_s^e$ is just the $e^{\rm th}$ iterate of the map induced by $t_s$. Therefore is it enough
to prove the assertion in the case $e=1$. On the other hand, this case follows if we show the surjectivity when $e$ is such that the cardinality of the residue field $k(s)$ is $p^e$. Note that in this case the map $t_s^e\colon H^0(X_s,\omega_{X_s}(E_s))\to H^0(X_s,\omega_{X_s}(E_s))$ is $k(s)$-linear.
Its Serre dual is the map $F^e\colon H^n(X_s,\cO_{X_s}(-E_s))\to H^0(X_s,\cO_{X_s}(-E_s))$, where 
$F$ denotes the Frobenius action on $H^n(X_s,\cO_{X_s}(-E_s))$. By assumtion, $F$ is semisimple
hence bijective, which implies the assertion in the lemma. 
\end{proof}

\begin{remark}\label{rem_cor2_lem_SNC}
It follows from the proofs of Lemma~\ref{lem_SNC} and of Corollaries~\ref{cor_lem_SNC}
and \ref{cor2_lem_SNC} that in order to get the assertions in these two corollaries we need to apply 
Conjecture~\ref{conj_introd} to finitely many smooth projective varieties. It follows from 
Remark~\ref{rem_ordinary2} that if we have finitely many pairs $(X^{(1)},E^{(1)}),\ldots,(X^{(m)},
E^{(m)})$ as in Corollaries~\ref{cor_lem_SNC} and \ref{cor2_lem_SNC}, then we can find a dense set of closed points $s\in\Spec A$ such that the conclusion in each of of these two corollaries holds for all these pairs. In particular, in Corollary~\ref{cor2_lem_SNC} we do not need to assume that 
$X$ is connected.
\end{remark}

We now turn to the relative setting, and state the main result of this section.

\begin{theorem}\label{thm_relative}
Suppose that Conjecture~\ref{conj_introd} holds.
Let $\pi\colon X\to T$ be a projective morphism of schemes over $k$, with $X$ nonsingular, and let
$E=E_1+\ldots+E_r$ be a reduced simple normal crossings divisor on $X$. If 
$\pi_A\colon X_A\to T_A$ and $E_A$ are models over $A$ for $\pi$ and $E$, respectively, then there is a dense set of closed points $S\subset\Spec A$ such that for every $e\geq 1$ and every $s\in S$, the induced morphism
\begin{equation}\label{eq_thm_relative}
(\pi_s)_*(F^e_*(\omega_{X_s}(E_s)))\to (\pi_s)_*(\omega_{X_s}(E_s))
\end{equation}
is surjective.
\end{theorem}

\begin{proof}
Suppose first that $T$ is affine. Since $\pi$ is projective, we have a closed immersion 
$X\hookrightarrow {\mathbf P}^N\times T$, for some $N\geq 1$. Let us fix an open immersion 
$T\hookrightarrow  T'$, where $T'$ is projective. Let $\overline{X}$ be the closure of $X$ in
${\mathbf P}^N\times T'$ (with the reduced scheme structure), and $\overline{\pi}\colon \overline{X}\to T'$ the induced morphism. Since $\overline{X}\cap ({\mathbf P}^N\times T)=X$, it follows 
that $\overline{\pi}^{-1}(T)=X$. 

By hypothesis,  $X$ is nonsingular and $E$ has simple normal crossings, hence by the standard results on resolution of singularities in characteristic zero, 
there is a projective morphism $\phi\colon X'\to \overline{X}$ that is an isomorphism over $X$,
with $X'$ nonsingular,
and a reduced simple normal crossings divisor $E'$ on $X'$ such that $E'\vert_{\phi^{-1}(X)}=
\phi^{-1}(E)$. If $\pi'=\overline{\pi}\circ\phi$, then $X$ is isomorphic to $(\pi')^{-1}(T)$, and it is clear that if the 
assertion in the theorem holds for $\pi'$ and $E'$, then it also holds for $\pi$ and $E$. Therefore we may assume that $X$ and $T$ are projective. 

We now choose a very ample line bundle $\cL$ on $T$ such that $\pi_*(\omega_X(E))\otimes \cL$
is globally generated. After possibly replacing $A$ by some
localization $A_a$ we may assume that for every closed point $s\in\Spec A$ we have
$(\omega_X)_s=\omega_{X_s}$ and
$\pi_*(\omega_X(E))_s=(\pi_s)_*(\omega_{X_s}(E_s))$. In particular, 
$(\pi_s)_*(\omega_{X_s}(E_s))\otimes\cL_s$ is globally generated. 

Since $\cL$ is very ample, the linear system $|\cL|$ and its pull-back to $X$ are globally generated. It follows from Bertini's Theorem (recall that 
${\rm char}(k)=0$) that if $D'\in |\cL|$ is a general element, then $E'=\pi^*(D')$ has the property that
$E+E'$ is a reduced simple normal crossings divisor. Of course, it is enough to ensure that
(\ref{eq_thm_relative}) is surjective after tensoring with $\cL_s$, and since 
$(\pi_s)_*(\omega_{X_s}(E_s))\otimes\cL_s$ is globally generated, it is enough to show that the map
\begin{equation}\label{eq2_thm_relative}
H^0(T_s,(\pi_s)_*(F^e_*(\omega_{X_s}(E_s)))\otimes \cL_s)\to H^0(T_s, (\pi_s)_*(\omega_{X_s}(E_s))
\otimes \cL_s)
\end{equation}
is surjective. 
By the projection formula, (\ref{eq2_thm_relative}) gets identified with the map
\begin{equation}\label{eq3_thm_relative}
H^0(X_s, F^e_*(\omega_{X_s}(E_s+p^eE'_s)))\to H^0(X_s, \omega_{X_s}(E_s+E'_s)).
\end{equation}
Applying Corollary~\ref{cor2_lem_SNC} to $X$ and $E+E'$, we deduce that there is a dense set of closed points $S\subseteq\Spec A$ such that the composition
$$H^0(X_s, F^e_*(\omega_{X_s}(E_s+E'_s)))\to H^0(X_s, F^e_*(\omega_{X_s}(E_s+p^eE'_s)))\to H^0(X_s, \omega_{X_s}(E_s+E'_s))$$
is surjective for every $s\in S$. This clearly implies the surjectivity of (\ref{eq3_thm_relative}), and completes
the proof in the case when $T$ is affine. 

Note that the proof in this special case relies on an application of Corollary~\ref{cor2_lem_SNC}
for one pair. In general, we consider a finite affine cover $T=\bigcup_iU_i$.
Combining what we proved so far with
Remark~\ref{rem_cor2_lem_SNC}, we see that 
 there is a dense set of closed points $S\subseteq\Spec A$
such that the assertion in the theorem holds for all morphisms $\pi^{-1}(U_i)\to U_i$ and for all
$s\in S$.  This implies the surjectivity
of the map in (\ref{eq_thm_relative}) for every $s\in S$, which completes the proof of the 
theorem.
\end{proof}

\begin{remark}\label{rem_thm_relative}
It follows from the proof of Theorem~\ref{thm_relative}, using also Remark~\ref{rem_cor2_lem_SNC},
that given finitely many morphisms $\pi^{(i)}\colon X^{(i)}\to T^{(i)}$ and divisors $E^{(i)}$ on $X^{(i)}$ satisfying
the hypothesis in the theorem, there is a dense set of closed points of $\Spec A$ 
that satisfies the conclusion of the theorem with respect to each of the morphisms $\pi^{(i)}$.
\end{remark}

\section{The connection between the two conjectures}\label{connection_conjectures}

We can now prove our main result, stated in the Introduction.

\begin{proof}[Proof of Theorem~\ref{thm_introd}]
Let us assume that Conjecture~\ref{conj_introd} holds. Actually,  we will use its consequence in 
Theorem~\ref{thm_relative}.
It follows from Proposition~\ref{prop_reduction_principal} that in order to show that
Conjecture~\ref{conj_introd2} holds, it is enough to consider the following setup.
Suppose that $Y$ is a nonsingular, irreducible affine variety over an algebraically closed field $k$ of 
characteristic zero. Let
$\fra=(f)$ be a nonzero principal ideal on $Y$. We need to show that given a model of
$(Y,\fra)$ over $A$, where $A$ is a subalgebra of $k$ of finite type over $\ZZ$, there is a dense
set of closed points $S\subset \Spec A$ such that
\begin{equation}\label{eq_main1}
\tau(Y_s,f_s^{\lambda})=\cJ(Y,f^{\lambda})_s
\end{equation}
for all $s\in S$ and all $\lambda\in\RR_{\geq 0}$. 
Furthermore, given finitely many such pairs $(Y,\fra)$, we need to be able to do this simultaneously for all the pairs.

After covering $Y$ by suitable affine open subsets, we may assume that $f\colon Y\to\AAA^1$
is smooth over $\AAA^1\smallsetminus\{0\}$. Indeed, there is an open neighborhood $U$ of $V(f)$ such that $f$ is smooth
on $U\smallsetminus V(f)$, while on $Y\smallsetminus V(f)$ we may replace $f$ by $1$.

Therefore we can apply the semi-stable reduction theorem of \cite{KKMS} for $f$ to get a 
positive integer 
$d\geq 1$
with the following property. If $\beta\colon\AAA^1\to\AAA^1$ is given by $\beta(t)=t^d$, 
and if we consider the Cartesian diagram
\begin{equation}\label{diag_proof_main}
\xymatrix{
W \ar[r]^{\alpha}\ar[d]_{g} &  Y\ar[d]^{f}\\
\AAA^1\ar[r]^{\beta} & \AAA^1
}
\end{equation}
then there is a projective morphism $\psi\colon Z\to W$ that satisfies
\begin{enumerate}
\item[(i)] $\psi$ is an isomorphism over $\AAA^1\smallsetminus\{0\}$ (in particular, $\psi$ is birational).
\item[(ii)] $Z$ is nonsingular.
\item[(iii)] $\psi^*(g)$ defines a \emph{reduced} simple normal crossings divisor on $Z$.
\end{enumerate}

Let $W_0$ be an irreducible component of $W$ that maps surjectively onto $Y$, and let
$X$ be the corresponding irreducible component of $Z$ that surjects onto $W_0$. 
If $Y'$ is the normalization of $W_0$, then we have induced morphisms
$$X\overset{\pi}\to Y'\overset{\phi}\to Y.$$ We denote by $h$ the pull-back of $g$ to $Y'$
By construction, $\phi$ is finite and surjective, \'{e}tale over $Y\smallsetminus
V(f)$. In particular, the singular locus of $Y'$ is contained in $V(h)$.

Let $D'=-K_{Y'/Y}$ be the divisor defined as in Proposition~\ref{invar_finite_mult}. Note that $D'$ is supported
on $V(h)$. It follows from 
Proposition~\ref{invar_finite_mult} (see also Remark~\ref{rem_invar_finite_mult}) 
that for every $\lambda\in\RR_{\geq 0}$ we have
\begin{equation}\label{eq11_main_thm}
\cJ(Y,f^{\lambda})=\phi_*\cJ(Y',D',h^{m\lambda})\cap \cO_Y.
\end{equation}

We define the divisor $D'_X$ as in \S\ref{mult_ideals}, such that in particular $K_X+D'_X$ and $\pi^*(K_{Y'}+D')$
are linearly equivalent. Let $E$ be the reduced simple normal crossings divisor defined on $X$ by
$\pi^*(h)$. Note that $D'_X$ is supported on $E$, which has simple normal crossings, hence it follows from
 Proposition~\ref{invar_birat_mult} and the definition of multiplier ideals that
 for every $\lambda\in\RR_{\geq 0}$ we have
 \begin{equation}\label{eq15_main_thm}
 \cJ(Y',D',h^{m\lambda})=\pi_*\cO_X(-D'_X-\lfloor m\lambda E\rfloor). 
 \end{equation}

We choose a model over a finitely generated $\ZZ$-algebra $A$, contained in $k$, for
all the above varieties and morphisms. We may assume that the above properties extend to
all fibers over $k(s)$, for $s\in\Spec A$ a closed point. Furthermore, after replacing $A$ by some localization $A_a$, we may assume that for every closed point $s\in\Spec A$ the characteristic
of $k(s)$ does not divide $[K(Y')\colon K(Y)]$. 
In this case Theorem~\ref{finite_test} applies to give
\begin{equation}\label{eq12_main_thm}
\tau(Y_s,f_s^{\lambda})=(\phi_s)_*\tau(Y'_s,D'_s,{h}_s^{m\lambda})\cap\cO_{Y_s}
\end{equation}
for every closed point $s\in\Spec A$ and every $\lambda\in\RR_{\geq 0}$. On the other hand,
we may assume that (\ref{eq11_main_thm}) induces 
\begin{equation}\label{eq13_main_thm}
\cJ(Y,f^{\lambda})_s=(\phi_s)_*\cJ(Y',D', h^{m\lambda})_s\cap \cO_{Y_s},
\end{equation}
and (\ref{eq15_main_thm}) induces
\begin{equation}\label{eq16_main_thm}
\cJ(Y',D',h^{m\lambda})_s=(\pi_s)_*\cO_{X_s}(-D'_{X_s}-\lfloor m\lambda E_s\rfloor)
\end{equation}
for every $s\in\Spec A$ and every $\lambda\in\RR_{\geq 0}$. 
Therefore in order to guarantee $\tau(Y_s, f_s^{\lambda})=\cJ(Y,f^{\lambda})_s$
for all $\lambda$, it is enough to ensure
\begin{equation}\label{eq17_main_thm}
\tau(Y'_s,D'_s, h_s^{m\lambda})=(\pi_s)_*\cO_{X_s}(-D'_{X_s}-\lfloor m\lambda E_s\rfloor)
\end{equation}
for all $\lambda\in\RR_{\geq 0}$. 

We now apply Theorem~\ref{thm_relative} to the morphism $\pi\colon X\to Y'$ and to the reduced
simple normal crossings divisor $E$. It follows that there is a dense set of closed points $S\subset\Spec A$ such that
$$(\pi_s)_*(F^e_*(\omega_{X_s}(E_s)))\to (\pi_s)_*(\omega_{X_s}(E_s))$$
is surjective for every $s\in S$ and every $e\geq 1$. The equality (\ref{eq17_main_thm}) now follows applying Lemma~\ref{main_ingredient}
below to the morphism $\pi_s\colon X_s\to Y'_s$, the divisor $D'_s$ and 
$h_s\in\Gamma(Y'_s,\cO_{Y'_s})$.
\end{proof}

\begin{lemma}\label{main_ingredient}
Let $\pi\colon X\to T$ be a birational morphism of schemes of finite type over a perfect field of 
characteristic $p>0$, with $T$ normal and irreducible, and $h\in\Gamma(T,\cO_T)$ nonzero. If
$D$ is a divisor on $T$ supported on $V(h)$ such that $K_T+D$
is Cartier, and if the following hold:
\begin{enumerate}
\item[(i)]  $X$ is nonsingular.
\item[(ii)] $E:=\pi^*({\rm div}(h))$ is a reduced divisor, with simple normal crossings.
\item[(iii)] $\pi$ is proper, and an isomorphism over $T\smallsetminus V(h)$.
\item[(iv)] The map $\pi_*(F^e_*(\omega_{X}(E)))\to \pi_*(\omega_{X}(E))$ is surjective
for every $e\geq 1$,
\end{enumerate}
then $\tau(T,D, h^{\lambda})=\pi_*\cO_X(-D_X-\lfloor \lambda E\rfloor)$ for every $\lambda\in
\RR_{\geq 0}$, where $D_X$ is defined as in \S\ref{mult_ideals}.
\end{lemma}

\begin{proof}
Note that by (iii), the divisor $D_X$ is supported on $E$, hence $D_X, E$ have simple
normal crossings by (ii). Proposition~\ref{one_inclusion} gives the inclusion ``$\subseteq$"
in the statement, hence we just need to show that
\begin{equation}\label{eq1_main_ingredient}
\pi_*\cO_X(-D_X-\lfloor\lambda E\rfloor)\subseteq\tau(T,D, h^{\lambda})
\end{equation}
for every $\lambda\geq 0$. 

After replacing $D$ by $D+m\cdot {\rm div}(h)$, with $m\gg 0$, we may assume that $D$ is effective.
It follows from the projection formula and from 
Proposition~\ref{Skoda_test} that it is enough to prove (\ref{eq1_main_ingredient})
for $\lambda<1$. 
Let us fix such $\lambda$. Note that in this case, the left-hand side of (\ref{eq1_main_ingredient})
is equal to $\pi_*\cO_X(-D_X)$. We write $D_X=\sum_{i=1}^Na_iE_i$.

After taking a finite affine open cover of $T$, we may assume that $T$ is affine.
For the description of $\tau(T,D, h^{\lambda})$ we use formula (\ref{eq2_def_test}). Note that 
by (i) and (iii), the singular locus of $T$ is contained in $V(h)$. Since $D$ is also supported on $V(h)$, we see that if $\ell\gg 0$, then we may take $c=h^{\ell}$ in formula (\ref{eq2_def_test}). 
We fix $\ell$ with this property such that, in addition,
$\ell\geq a_i$ for all $i$.
It follows that it is enough 
to show that if $e\gg 0$, then
\begin{equation}\label{eq2_main_ingredient}
\pi_*\cO_X(-D_X)\subseteq \phi_D^{(e)}\left(F^e_*(h^{d_e}\cdot\cO_T((1-p^e)(K_T+D)))\right),
\end{equation}
where $d_e=\lceil\lambda(p^e-1)\rceil+\ell$. 

For the sake of a more compact notation, let us put $\cL=\cO_T((1-p^e)(K_T+D))$.
We use the commutative diagram (\ref{diag3_3}) to write the right-hand side of (\ref{eq2_main_ingredient}) as
\begin{equation}\label{eq4_main_ingredient}
\pi_*(\phi_{D_X}^{(e)})\left(F^e_*(\rho(h^{d_e}\cdot\cL))\right),
\end{equation}
in which we recall that $\rho\colon \cL\to \pi_*(\pi^*(\cL))$ denotes the canonical isomorphism.
It is clear that $F^e_*(\rho(h^{d_e}\cdot\cL))=\pi_*\left(F^e_*(h^{d_e}\cdot\pi^*(\cL))\right)$.
For $e\gg 0$, $\phi_{D_X}^{(e)}$ induces a surjection on $X$
\begin{equation}\label{eq5_main_ingredient}
u\colon F^e_*(h^{d_e}\cdot\pi^*(\cL))\to \cO_X(-D_X)
\end{equation}
This follows from Example~\ref{computation_SNC} and the fact that
$\lfloor (d_e-a_i)/p^e\rfloor=0$ for $e\gg 0$. 

\noindent{\bf Claim}. $\pi_*(u)\colon \pi_*\left(F^e_*(h^{d_e}\cdot\pi^*(\cL))\right)
\to\pi_*\cO_X(-D_X)$ is surjective.

If this holds, then the expression in (\ref{eq4_main_ingredient}) 
is equal to $\pi_*\cO_X(-D_X)$, which gives the inclusion in (\ref{eq2_main_ingredient}).

Therefore the proof of the  lemma is complete if we show the Claim. Note that the surjectivity 
of $\pi_*(u)$ is equivalent to the surjectivity of $\pi_*(u)\otimes\cO_T(K_T+D)$. Using the projection
formula, this becomes equivalent to the surjectivity of
\begin{equation}\label{eq7_main_ingredient}
\pi_*\left(F^e_*(\omega_X(D_X-d_eE))\right) \to\pi_*(\omega_X).
\end{equation}
For $e\gg 0$, the divisor $(D_X-d_eE)+(p^e-1)E$ is effective, hence the surjectivity
of the map in (\ref{eq7_main_ingredient}) follows from the surjectivity of 
\begin{equation}\label{eq8_main_ingredient}
w\colon \pi_*\left(F^e_*(\omega_X(-(p^e-1)E))\right) \to\pi_*(\omega_X).
\end{equation}
This in turn is surjective if and only if $w\otimes\cO_T(H)$ is surjective, but the latter map is identified
via the projection formula with 
$$\pi_*\left(F^e_*(\omega_X(E))\right)\to\pi_*(\omega_X(E)),$$
which is surjective by the assumption in (iv). This completes the proof of the lemma.
\end{proof}

\providecommand{\bysame}{\leavevmode \hbox \o3em
{\hrulefill}\thinspace}

\end{document}